\documentclass[11pt, a4paper]{amsart}

\usepackage{amsmath,amsthm,amssymb,amscd,fullpage}
 
\theoremstyle{plain}
\newtheorem{theorem}{Theorem}[section]
\newtheorem{lemma}[theorem]{Lemma}
\newtheorem{claim}{Claim} 
\newtheorem{proposition}[theorem]{Proposition}

\newtheorem{mainth}{Theorem}

\theoremstyle{definition}
\newtheorem{definition}[theorem]{Definition}
\newtheorem{notation}[theorem]{Notation}

\newtheorem*{acknowledgement}{Acknowledgement}

\theoremstyle{remark}
\newtheorem{remark}[theorem]{Remark}
\newtheorem{fact}[theorem]{Fact}

\numberwithin{equation}{section}

\newcommand{\bP}{\mathbb{P}}
\newcommand{\sP}{\mathsf{P}}
\newcommand{\cS}{\mathcal{S}}
\newcommand{\bN}{\mathbb{N}}
\newcommand{\sJ}{\mathsf{J}}
\newcommand{\supp}{\operatorname{supp}}
\newcommand{\cE}{\mathcal{E}}
\newcommand{\sH}{\mathsf{H}}
\newcommand{\cO}{\mathcal{O}}
\newcommand{\bD}{\mathbb{D}}
\newcommand{\bC}{\mathbb{C}}
\newcommand{\bL}{\mathbb{L}}
\newcommand{\bR}{\mathbb{R}}
\newcommand{\diam}{\operatorname{diam}}
\newcommand{\cM}{\mathcal{M}}
\newcommand{\GCD}{\mathrm{GCD}}
\newcommand{\PGL}{\mathrm{PGL}}
\newcommand{\ord}{\operatorname{ord}}
\newcommand{\bZ}{\mathbb{Z}}
\newcommand{\sR}{\mathsf{R}}
\newcommand{\Char}{\operatorname{char}}
\newcommand{\crit}{\operatorname{Crit}}
\newcommand{\widevec}[1]{\overrightarrow{#1}}

\begin{document}
\title[On a degenerating limit theorem of DeMarco--Faber]{On a degenerating limit theorem of DeMarco--Faber}

\author{Y\^usuke Okuyama}
\address{Division of Mathematics, Kyoto Institute of Technology, Sakyo-ku, Kyoto 606-8585 JAPAN}
\email{okuyama@kit.ac.jp}

\date{\today}


\begin{abstract}
 One of our aims is to complement the proof of DeMarco--Faber's 
 degenerating limit theorem for
 the family of the unique maximal entropy measures
 parametrized by a punctured open disk 
 associated to a meromorphic family of rational functions
 on the complex projective line degenerating at the puncture. 
 This complementation is done by our main result,
 which rectifies a key computation in their argument.
 We also establish and use a direct and explicit translation from 
 degenerating complex dynamics into quantized Berkovich dynamics,
 instead of using DeMarco--Faber's more conceptual transfer principle
 between those dynamics.
\end{abstract}

\maketitle 

\section{Introduction}
\label{sec:intro}
Let $K$ be an algebraically closed field that is complete with respect to
a non-trivial and non-archimedean absolute value.
The action of a rational function $h\in K(z)$ 
on $\bP^1=\bP^1(K)$ extends continuously to that on 
the Berkovich projective line $\sP^1=\sP^1(K)$, which is a compact
augmentation of $\bP^1$. If in addition
$\deg h>0$, then this extended action of $h$
on $\sP^1$ is surjective, open, and
fiber-discrete and preserves the type (among I, II, III, and IV)
of each point in $\sP^1$,
and the local degree function $\deg_{\,\cdot}h$ of $h$ 
on $\bP^1$ also extends upper semicontinuously to $\sP^1$ so that for every open subset $V$ in $\sP^1$ and every component $U$ 
of $h^{-1}(V)$, $V\ni\cS'\mapsto\sum_{\cS\in h^{-1}(\cS')\cap U}\deg_{\cS}h\equiv\deg(h:U\to V)$. 

The pushforward operator $h_*:C^0(\sP^1)\to C^0(\sP^1)$ 
is defined so that for every $\psi\in C^0(\sP^1)$,
$(h_*\psi)(\cdot):=\sum_{\cS\in h^{-1}(\cdot)}(\deg_{\cS}h)\psi(\cS)$
on $\sP^1$. 
The pullback operator $h^*$ from the
space $M(\sP^1)$ of all Radon measures on $\sP^1$ to itself is defined by the 
transpose of $h_*$, so that for every $\nu\in M(\sP^1)$,
\begin{gather}
 h^*\nu=\int_{\sP^1}\Bigl(\sum_{\cS'\in h^{-1}(\cS)}(\deg_{\cS'}h)\delta_{\cS'}\Bigr)\nu(\cS)\quad\text{on }\sP^1,\label{eq:pullbackdirac}
\end{gather} 
where for each point $\cS\in\sP^1$, $\delta_{\cS}$ is the Dirac measure
at $\cS$ on $\sP^1$; in particular, $(h^*\delta_{\cS})(\sP^1)=\deg h$.

\subsection{Factorization on $\sP^1$ and quantization}\label{sec:skeltal}
We follow the presentation in \cite[\S 4.2]{DF14}.
For each finite subset $\Gamma$ consisting of type II points
(e.g., a semistable vertex set) in $\sP^1$, the family
\begin{gather*}
 S(\Gamma):=
\bigl\{\text{either a component of }\sP^1\setminus\Gamma\text{ or a singleton }\{\cS\}\text{ for some }\cS\in\Gamma\bigr\}
\subset 2^{\sP^1}
\end{gather*}
is a partition of $\sP^1$; 
the measurable factor space $\sP^1/S(\Gamma)=S(\Gamma)$ equipped with
the $\sigma$-algebra $2^{S(\Gamma)}$ is regarded as 
the measurable space $(\sP^1,2^{S(\Gamma)})$, 
also regarding $2^{S(\Gamma)}$
as a $\sigma$-subalgebra in the Borel $\sigma$-algebra on $\sP^1$.

Let $M(\Gamma)$ be the set of all complex measures $\omega$
on $\sP^1/S(\Gamma)$.
The measurable factor map
\begin{gather*}
\pi_{\Gamma}=\pi_{\sP^1,\Gamma}:\sP^1\to\sP^1/S(\Gamma) 
\end{gather*}
induces the pullback operator $(\pi_\Gamma)^*$ from the space of 
measurable functions on $\sP^1/S(\Gamma)$ to that of measurable functions on $\sP^1$
and, in turn, the transpose (projection/quantization operator) 
$(\pi_\Gamma)_*:M(\sP^1)\to M(\Gamma)$ of $(\pi_\Gamma)^*$ 
(by restricting each element of $M(\sP^1)$ to $2^{S(\Gamma)}$),
so in particular that for every $\nu\in M(\sP^1)$,
\begin{gather}
 \bigl((\pi_\Gamma)_*\nu\bigr)(\{U\})=\nu(U)
\quad\text{for any }U\in S(\Gamma).\label{eq:projectfactor}
\end{gather}
Set
$M^1(\sP^1):=\{\omega\in M(\sP^1):\omega\ge 0\text{ and }\omega(\sP^1)=1\}$
and 
$M^1(\Gamma):=\{\omega\in M(\Gamma):\omega\ge 0\text{ and }\omega(\sP^1/S(\Gamma))=1\}$, so that $(\pi_\Gamma)_*(M^1(\sP^1))\subset M^1(\Gamma)$. 
Set also
\begin{gather*}
M^1(\Gamma)^\dag:=\bigl\{\omega\in M^1(\Gamma):\omega(\{\cS\})=0
\text{ for every }\cS\in\Gamma\bigr\}.
\end{gather*}
For any finite subsets $\Gamma$ and $\Gamma'$, $\Gamma\subset\Gamma'$, both
consisting of type II points, the measurable factor map 
\begin{gather*}
 \pi_{\Gamma',\Gamma}:\sP^1/S(\Gamma')\to\sP^1/S(\Gamma) 
\end{gather*}
induces the pullback operator $(\pi_{\Gamma',\Gamma})^*$ 
from the space of measurable functions on $\sP^1/S(\Gamma)$
to that of measurable functions on $\sP^1/S(\Gamma')$ (so that
$\pi_\Gamma^*=(\pi_{\Gamma'})^*(\pi_{\Gamma',\Gamma})^*$)
and, in turn, the transpose (or projection operator) $(\pi_{\Gamma',\Gamma})_*:M(\Gamma')\to M(\Gamma)$ of $(\pi_{\Gamma',\Gamma})^*$, so in particular that for every $\omega\in M(\Gamma')$,
\begin{gather} 
 \bigl((\pi_{\Gamma',\Gamma})_*\omega\bigr)(\{U\})
=\omega\bigl(\bigl\{V\in S(\Gamma'):V\subset U\bigr\}\bigr)
\quad\text{for any }U\in S(\Gamma)\label{eq:factorrel} 
\end{gather}
and that $(\pi_{\Gamma',\Gamma})_*(\pi_{\Gamma'})_*=(\pi_{\Gamma})_*$.
Then
$(\pi_{\Gamma',\Gamma})_*(M^1(\Gamma')^\dag)\subset M^1(\Gamma)^\dag$.

Let us denote by $\cS_G$ the Gauss (or canonical) point in $\sP^1$,
which is a type II point (see Subsection \ref{sec:berkovichgeneral}).
For a rational function $h\in K(z)$ on $\bP^1$ of degree $>0$, 
noting that $h(\cS_G)$ is also a type II point and
setting
\begin{gather*}
 \Gamma_G:=\{\cS_G\}\quad\text{and}\quad\Gamma_h:=\{\cS_G,h(\cS_G)\}, 
\end{gather*}
the quantized pullback operator $h_G^*:M(\Gamma_h)\to M(\Gamma_G)$  
is induced from the pullback operator $h^*
$ in \eqref{eq:pullbackdirac}; for every $\omega\in M(\Gamma_h)$,
the measure $h_G^*\omega\in M(\Gamma_G)$ in particular satisfies
\begin{gather*}
 \bigl(h_G^*\omega\bigr)(\{U\})=\int_{\sP^1/S(\Gamma_h)}m_{V,U}(h)\omega(V)
\quad\text{for any }U\in S(\Gamma_G),
\end{gather*}
where the quantized local degree $m_{V,U}(h)$ of $h$ with respect to
each pair $(U,V)\in S(\Gamma_G)\times S(\Gamma_h)$ 
is induced from the local degree function $\deg_{\,\cdot}h$
on $\sP^1$ so that, fixing any $\cS'\in V$,
\begin{align*}
 m_{V,U}(h)
=\begin{cases}
 (h^*\delta_{\cS'})(U)
  & \text{if }U\in S(\Gamma_G)\setminus\{\{\cS_G\}\}\text{ and }V\in S(\Gamma_h)\setminus\{\{h(\cS_G)\}\},\\
 (h^*\delta_{\cS'})(\{\cS_G\})  & \text{if }U=\{\{\cS_G\}\}
 \end{cases}
\end{align*}
(the remaining case that $U\in S(\Gamma_G)\setminus\{\{\cS_G\}\}$
and
$V=\{\{h(\cS_G)\}\}$ is more subtle) and that for every $V\in S(\Gamma_h)$,
$\sum_{U\in S(\Gamma_G)}m_{V,U}(h)=\deg h$. In particular,
\begin{gather*}
 \bigl(h_G^*\omega\bigr)(S(\Gamma_G))
=(\deg h)\cdot\omega(S(\Gamma_h))\quad\text{for every }\omega\in M(\Gamma_h),\quad\text{and}\\
\bigl((\deg h)^{-1}h_G^*\bigr)(M^1(\Gamma_h)^\dag)\subset M^1(\Gamma_G)^\dag
\end{gather*}
(see Subsection \ref{sec:quantized} for more details including
the precise definition of $m_{V,U}(h)$ and that of $h_G^*$). 

\subsection{The $f$-balanced measures on $\sP^1$ and the maximally ramification locus of $f$
in $\sP^1$}\label{sec:balanced}
From now on, let $f\in K(z)$ be a rational function on $\bP^1$ of $\deg f=:d>1$. 

The equilibrium (or canonical) measure $\nu_f$ of $f$ on $\sP^1$ is 
the weak limit
\begin{gather}
 \nu_f:=\lim_{n\to\infty}\frac{(f^n)^*\delta_{\cS}}{d^n}\quad\text{in }M(\sP^1)\quad\text{for any }\cS\in\sP^1\setminus E(f)\label{eq:equidist}
\end{gather}
(see \cite{FR09} for the details), and is 
the unique $\nu\in M^1(\sP^1)$ not only having the $f$-balanced property 
\begin{gather*}
 f^*\nu=(\deg f)\cdot\nu\quad\text{on }\sP^1
\end{gather*}
but also satisfying the vanishing condition $\nu(E(f))=0$. Here
the (classical) exceptional set
$E(f):=\{a\in\bP^1:\#\bigcup_{n\in\bN\cup\{0\}}f^{-n}(a)<\infty\}$
of $f$ is the union of all (superattracting) cycles of $f$ in $\bP^1$
totally invariant under $f$ (so 
is at most countable).

The ramification locus $\sR(f):=\{\cS\in\sP^1:\deg_{\cS}f>1\}$ of $f$
contains 
the (classical) critical set 
$\crit(f):=\{c\in\bP^1:f'(c)=0\}$ of $f$, and
the maximally ramification locus 
\begin{gather*}
 \sR_{\max}(f):=\{\cS\in\sP^1:\deg_{\cS}(f)=d\}(\subset\sR(f)) 
\end{gather*}
of $f$ contains $E(f)(\subset\crit(f))$;
since $\sR_{\max}(f)$ is connected (Faber \cite[Theorem 8.2]{Faber13topologyI}),
for every 
$c\in\sR_{\max}(f)\cap\bP^1$, 
$\sR_{\max}(f)$ near $c$ contains a closed {\itshape interval} $[c,\cS]$
(see Subsection \ref{sec:berkovichgeneral}) 
in $\sP^1$ for some $\cS\in\sP^1\setminus\{c\}$. 

\begin{definition}[tame maximal ramification]
For each $c\in\sR_{\max}(f)\cap\bP^1$,
we say $f$ is tamely maximally ramified near $c$
if $\sR_{\max}(f)$ near $c$ {\em is} a closed interval $[c,\cS]$
in $\sP^1$ for some $\cS\in\sP^1\setminus\{c\}$. 
\end{definition}

\begin{fact}[a consequence of Faber {\cite[Corollary 6.6]{Faber13topologyI}}]
 $f$ is tamely maximally ramified at every $c\in\sR_{\max}(f)\cap\bP^1$
 if 
%$\Char K=0$ and 
the residue characteristic 
 of $K$ is either $=0$ or $>d(=\deg f)$ (e.g.\ when $K=\bL$
 as in Subsection \ref{sec:deglimit} below).
\end{fact}

We note that when $\Char K=0$, 
\begin{gather}
 E(f)=\bigl\{a\in\bP^1:f^{-2}(a)=\{a\}\bigr\}\quad\text{and}\quad
\#E(f)\le\#\bigl(\sR_{\max}(f)\cap\bP^1\bigr)
\le 2.\label{eq:maximaldegree}
\end{gather}

The Berkovich Julia set $\sJ(f):=\supp\nu_f$ of $f$ is in $\sP^1\setminus E(f)$ 
(by \eqref{eq:equidist}); 
both $\sJ(f)$ and $E(f)$ are 
$f$-completely invariant.
Any $\nu\in M^1(\sP^1)$ (only) having the above $f$-balanced property on $\sP^1$
is written as 
\begin{gather*}
\nu=\nu(\sJ(f))\cdot\nu_f+\sum_{\cE\subset E(f):\text{ a cycle of }f}\nu(\cE)\cdot\frac{\sum_{a\in\cE}\delta_a}{\#\cE}\quad\text{on }\sP^1
\end{gather*}
(by \eqref{eq:equidist} and the countability of $E(f)$). For every $n\in\bN$, 
we also have $\nu_{f^n}=\nu_f$ in $M^1(\sP^1)$ (so $\sJ(f^n)=\sJ(f)$) 
and $E(f^n)=E(f)$. 

Recall that for any $\cS\in\sH^1:=\sP^1\setminus\bP^1$,
\begin{gather}
 f^{-1}(\cS)\neq\{\cS\}\Leftrightarrow\nu_f(\{\cS\})<1\Leftrightarrow\supp(\nu_f)\neq\{\cS\}
\Leftrightarrow\nu_f(\{\cS\})=0\Leftrightarrow\nu_f(\{f(\cS)\})=0\label{eq:nopotgood}
\end{gather}
(see e.g.\ {\cite[Corollary 10.33]{BR10}}),
so in particular, $f^{-1}(\cS)\neq\{\cS\}$ if and only if 
$f^{-n}(\cS)\neq\{\cS\}$ for every $n\in\bN$.
For 
every $\nu\in M^1(\sP^1)$ having the $f$-balanced property on $\sP^1$
and
every finite subset $\Gamma$ in $\sP^1$ consisting of type II points, we have
\begin{gather}
 \bigl((\pi_{\Gamma})_*\nu\bigr)\bigl(S(\Gamma)\setminus F\bigr)=0\quad
\text{for some countable subset }F\text{ in }S(\Gamma)\label{eq:countablysupp}
\end{gather}
(by \eqref{eq:equidist}) and
\begin{gather}
 (\pi_{\Gamma})_*\nu\in M^1(\Gamma)^\dag
\quad\text{if in addition }f^{-1}(\cS)\neq\{\cS\}\text{ for every }\cS\in\Gamma.\label{eq:canonicalpush}
\end{gather}

\subsection{Main result: the projections of the $f$-balanced measures
on $\sP^1$ to $\sP^1/S(\Gamma_G)$}\label{sec:mainth} 
Recall that $d:=\deg f>1$ and that
$\Gamma_G:=\{\cS_G\}$, and for each $n\in\bN$, set
\begin{gather*}
 \Gamma_n:=\Gamma_{f^n}=\bigl\{\cS_G,f^n(\cS_G)\bigr\}.
\end{gather*}

Let us say $\omega\in M^1(\Gamma_f)$ has
the quantized $f$-balanced property if
\begin{gather}
 f_G^*\omega=d\cdot (\pi_{\Gamma_f,\Gamma_G})_*\omega\quad\text{in }M^1(\Gamma_G).\label{eq:quantizedbalanced} 
\end{gather}
Set $\Delta_f\subset M^1(\Gamma_G)$ 
(resp.\ $\Delta_f^\dag\subset M^1(\Gamma_G)^\dag$) by
\begin{multline}
\Delta_f\text{ (resp.\ $\Delta_f^\dag$)}\\
:=\Bigl\{\omega\in M^1(\Gamma_G):
\text{for (any) }n\gg 1,\text{ there is }\omega_n\in M^1(\Gamma_n)
\text{ (resp.\ $\omega_n\in M^1(\Gamma_n)^\dag$)}\\
\text{such that }
\omega_n\bigl(S(\Gamma_n)\setminus F\bigr)=0
\text{ for some countable subset }F\text{ in }S(\Gamma_n)\text{ and}\\
\text{that }d^{-n}\bigl((f^n)_G\bigr)^*\omega_n=
\omega=(\pi_{\Gamma_n,\Gamma_G})_*\omega_n\text{ in }M^1(\Gamma_G)\Bigr\};
\label{eq:deltaquantized}
\end{multline}
for a subtlety on the first vanishing assumption on each $\omega_n$,
see Remark \ref{th:notalways}.

Our principal result is the following computations of 
$\Delta_f$ (and $\Delta_f^\dag$) when $\Char K=0$,
which in particular rectifies \cite[Theorem 4.10, Corollary 4.13]{DF14};
the assumption on the period of each $a\in E(f)$
is for simplicity, and $f^2$ always satisfies this condition, and
the tame maximal ramification condition for $f$ near $a$ in the case (ii)
to obtain \eqref{eq:computationdelta}
always holds when $K=\bL$ as in Subsection \ref{sec:deglimit} below. 

\begin{mainth}\label{th:computation}
Let $K$ be an algebraically closed field of characteristic $0$
that is complete with respect to
a non-trivial and non-archimedean absolute value, 
let $f\in K(z)$ be a rational function on $\bP^1$ of degree $d>1$,
and suppose that $f^{-1}(\cS_G)\neq\{\cS_G\}$ and
that $f(a)=a$ $($or equivalently $f^{-1}(a)=\{a\}$$)$ for any $a\in E(f)$.
Then one and only one of the following cases $(i)$ and $(ii)$ occurs$;$

$(i)$ $\Delta_f=\Delta_f^\dag=\{(\pi_{\Gamma_G})_*\nu_f\}$.

$(ii)$ 
there is a $($unique$)$ $a\in E(f)$ such that $\lim_{n\to\infty}f^n(\cS_G)=a$
and that $f^n(\cS_G)$ is in the interval $(\cS_G,a]$ in $\sP^1$
for $n\gg 1$, and 
then $\deg_{f^n(\cS_G)}(f)\equiv d$ 
$($i.e., $f^n(\cS_G)\in\sR_{\max}(f)$$)$
for $n\gg 1$, 
and $\{(\pi_{\Gamma_G})_*\nu_f\}\subsetneq\{(\pi_{\Gamma_G})_*\delta_a,(\pi_{\Gamma_G})_*\nu_f\}\subset\Delta_f^\dag$.

In the case $($ii$)$, if in addition 
$f$ is tamely maximally ramified near 
%the 
$a$, then
\begin{multline}
\Delta_f=\biggl\{\omega\in M^1(\Gamma_G):\text{satisfying }
\begin{cases}
\omega(\{U_{\vec{v}}\})=s\nu_f(U_{\vec{v}})\\
\hspace*{40pt}\text{for every }\vec{v}\in\bigl(T_{\cS_G}\sP^1\bigr)\setminus\{\widevec{\cS_Ga}\},\\
 \omega(\{\{\cS_G\}\})=s',\text{ and}\\
\omega(\{U_{\widevec{\cS_Ga}}\})
=\Bigl(s\nu_f\bigl(U_{\widevec{\cS_Ga}}\bigr)+(1-s)\Bigr)-s'
\end{cases}\\
\hspace*{10pt}\text{for some }
s\in[0,1]\text{ and some }
s'\in\Bigl[0,\min\bigl\{s\nu_f\bigl(U_{\widevec{\cS_Ga}}\bigr),(1-s)\bigl(1-\nu_f\bigl(U_{\widevec{\cS_Ga}}\bigr)\bigr)\bigr\}\Bigr]
\biggr\},
\label{eq:computationdelta}
\end{multline}
which in particular yields
\begin{gather*}
 \Delta_f^\dag=
\bigl\{s\cdot(\pi_{\Gamma_G})_*\nu_f+(1-s)\cdot(\pi_{\Gamma_G})_*\delta_a:s\in[0,1]\bigr\},
\end{gather*}
and moreover, the following three statements that
$\deg_{f^n(\cS_G)}(f)\equiv d$ $($i.e., $f^n(\cS_G)\in\sR_{\max}(f)$$)$
for any $n\in\bN\cup\{0\}$, that
$\nu_f(U_{\widevec{\cS_Ga}})=0$, and that $\Delta_f=\Delta_f^\dag$
are equivalent.
\end{mainth}

In the proof of Theorem \ref{th:computation}, we will also
point out that for some $f$ (indeed $f(z)=z^2+t^{-1}z\in(\cO(\bD)[t^{-1}])[z](\subset\bL[z])$ and its iterations),
we have the proper inclusion $\Delta_f^\dag\subsetneq\Delta_f$.

\subsection{Application: the degenerating weak limit for the maximal entropy measures on $\bP^1(\bC)$}\label{sec:deglimit}

We call an element $f\in(\cO(\bD)[t^{-1}])(z)$ of degree say $d\in\bN\cup\{0\}$
a meromorphic family of rational functions on 
$\bP^1(\bC)$ (of degree $d$ and parametrized by 
\begin{gather*}
 \bD=\{t\in\bC:|t|<1\})
\end{gather*}
if for every $t\in\bD^*=\bD\setminus\{0\}$, 
the specialization $f_t$ of $f$ at $t$ is 
a rational function on $\bP^1(\bC)$ 
of degree $d$. Let us denote by $\mathbb{L}$ 
the (algebraically closed and complete) 
valued field of formal Puiseux series$/\bC$ 
around $t=0$, i.e., the completion of the field $\overline{\bC((t))}$ 
of Puiseux series$/\bC$ around $t=0$
valuated by their vanishing orders at $t=0$.
Noting that $\cO(\bD)[t^{-1}]$ 
is a subring of the field $\bC((t))$ of Laurent series$/\bC$ around 
$t=0$, we also regard $f$ as an element of $\bL(z)$. 
If in addition $d>1$, then for every $t\in\bD^*$, 
there is the equilibrium (or canonical, and indeed the unique maximal entropy) 
measure $\mu_{f_t}$ of $f_t$ on $\bP^1(\bC)$ (see Fact \ref{maxentropy}). 
As already seen in Subsection \ref{sec:balanced}, 
there is also the equilibrium (or canonical) measure $\nu_f$ of 
the $f\in\bL(z)$ of degree $d>1$ on $\sP^1(\bL)$.

If in addition $\nu_f(\{\cS_G\})=0$ or equivalently 
$f^{-1}(\cS_G)\neq\{\cS_G\}$ in $\sP^1(\bL)$
(mentioned in \eqref{eq:nopotgood}, see also another equivalent 
condition \eqref{eq:goodred} below),
then recalling that $\Gamma_G:=\{\cS_G\}$ 
as in Subsection \ref{sec:skeltal} and noting that
\begin{gather*}
 S(\Gamma_G)\setminus\{\{\cS_G\}\}=T_{\cS_G}(\sP^1(\bL))\cong
\bP^1(k_{\bL})=\bP^1(\bC), 
\end{gather*}
where $k_{\bL}(=\bC$ as fields) is the residue field of $\bL$ and 
where the bijection
between the tangent (or directions) space $T_{\cS_G}(\sP^1(\bL))$ of $\sP^1(\bL)$ at $\cS_G$ and $\bP^1(k_{\bL})$ is given by 
$\widevec{\cS_G a}\leftrightarrow\tilde{a}$ for each $a\in\bP^1(\bL)$
(see Subsection \ref{th:constantreduction} for the reduction 
$\tilde{a}\in\bP^1(k_{\bL})$ of $a$), 
the projection/quantization 
$(\pi_{\Gamma_G})_*\nu_f\in M^1(\Gamma_G)^{\dag}$
of $\nu_f\in M_1(\sP^1(\bL))$ is 
also regarded as a purely atomic probability measure on $\bP^1(\bC)$
(by \eqref{eq:canonicalpush}).

Using Theorem \ref{th:computation} and by some new arguments 
relating the $t$-adic absolute value on $\bL$, which 
restricts to the trivial 
(so non-archimedean) absolute value on $\bC=k_{\bL}$,
with the 
%(archimedean and non-trivial) 
Euclidean absolute value on $\bC$,
we complement the proof of the following degenerating
limit theorem of DeMarco--Faber. 

\begin{mainth}[{\cite[Theorem B]{DF14}}]\label{th:B}
For every meromorphic family 
\begin{gather*}
 f\in\bigl(\cO(\bD)[t^{-1}]\bigr)(z)(\subset\bL(z)) 
\end{gather*}
of rational functions
on $\bP^1(\bC)$ of degree $>1$, if $f^{-1}(\cS_G)\neq\{\cS_G\}$
in $\sP^1(\bL)$, then 
\begin{gather}
 \lim_{t\to 0}\mu_{f_t}=(\pi_{\Gamma_G})_*\nu_f\quad\text{weakly on }\bP^1(\bC).\label{eq:convergence}
\end{gather} 
\end{mainth}

We dispense with the intermediate 
``target bimeromorphically modified surface dynamics'' part
in the (conceptual) ``transfer principle'' from degenerating
complex dynamics to quantized Berkovich dynamics
in \cite[The proof of Theorem B]{DF14}, and give and use
a more direct and explicit translation from degenerating 
complex dynamics into quantized Berkovich dynamics
(see Definition \ref{th:direct} and Proposition \ref{th:transfer}
in Section \ref{sec:direct}). 
We hope our argument could also be helpful for a further investigation
of degenerating complex dynamics (see e.g.\ \cite{Favre16,DF18}).

\subsection*{Organization of the paper}
In Sections \ref{sec:background} and \ref{sec:quantizedbalanced},
we recall some notion and facts from non-archimedean dynamics on $\sP^1$ and
also recall some details on DeMarco--Faber's degenerating balanced property 
for degenerating weak limit points of the maximal entropy measures on $\bP^1(\bC)$,
respectively.
Section \ref{sec:direct} is one of the main parts in this paper, as mentioned 
in the above paragraph.
Theorem \ref{th:computation} is shown in Section \ref{sec:proofmain}, and
our proof of Theorem \ref{th:B} is given in Section \ref{sec:proof}. 
In Section \ref{sec:example}, a specific example, which motivated
our computation of $\Delta_f$ (and $\Delta_f^\dag$) in Theorem \ref{th:computation},
is discussed. In Section \ref{sec:complement}, we 
develop a little further our
direct translation from degenerating 
complex dynamics into quantized Berkovich dynamics, for completeness.

\section{Background from Berkovich dynamics}
\label{sec:background}

Let $K$ be an algebraically closed field that is complete with respect to
a non-trivial and non-archimedean absolute value $|\cdot|$.

\subsection{Berkovich projective line}\label{sec:berkovichgeneral}
We call $B(a,r):=\{z\in K:|z-a|\le r\}$ for some $a\in K$ and 
some $r\in\bR_{\ge 0}$ a $K$-closed disk; for any $K$-closed disks
$B,B'$, if $B\cap B'\neq\emptyset$, then either $B\subset B'$ or $B\supset B'$.
The Berkovich projective line $\sP^1=\sP^1(K)$ over $K$ is
a compact, uniquely arcwise connected, 
locally arcwise connected, and Hausdorff topological space;
as sets, 
\begin{gather*}
\sP^1=\bP^1\cup\sH^1=\bP^1\cup\sH^1_{\mathrm{II}}\cup\sH^1_{\mathrm{III}}\cup\sH^1_{\mathrm{IV}}\quad (\text{the disjoint unions}),\\
\bP^1=\bP^1(K)=K\cup\{\infty\}\cong\sP^1_{\mathrm{I}}\cong\bigl\{\{a\}=B(a,0):a\in K\bigr\}\cup\{\infty\},\\
\sH^1_{\mathrm{II}}\cong\bigl\{B(a,r):a\in K,r\in|K^*|\bigr\},\quad\text{and}\\
\sH^1_{\mathrm{III}}\cong\bigl\{B(a,r):a\in K,r\in\bR_{>0}\setminus|K^*|\bigr\}.
\end{gather*}
More precisely, each element of $\sP^1$ is regarded as
either the cofinal equivalence 
class of a decreasing (i.e., non-increasing and nesting)
sequence of $K$-closed disks or $\infty\in\bP^1$. 
The inclusion relation $\subset$ among $K$-closed disks
canonically
extends to an ordering $\preceq$ on $\sP^1$, so that $\infty$ is 
the maximum element in $(\sP^1,\preceq)$, and the diameter function
$\diam_{|\cdot|}$ for $K$-closed disks also extends 
upper semicontinuously
to $\sP^1$, so that $\diam_{|\cdot|}(\infty)=+\infty$. 
For $\cS_1,\cS_2\in\sP^1$, if $\cS_1\preceq\cS_2$, then we set
$[\cS_1,\cS_2]=[\cS_2,\cS_1]:=\{\cS\in\sP^1:\cS_1\preceq\cS\preceq\cS_2\}$, and
in general there is the minimum element $\cS'$ in $\{\cS\in\sP^1:\cS_1\preceq\cS\text{ and }\cS_2\preceq\cS\}$ and we set 
\begin{gather*}
 [\cS_1,\cS_2]=[\cS_2,\cS_1]:=[\cS_1,\cS']\cup[\cS',\cS_2];
\end{gather*}
we also set $(\cS_1,\cS_2]:=[\cS_1,\cS_2]\setminus\{\cS_1\}$.
Those (closed) intervals $[\cS,\cS']$
in $\sP^1$ equip $\sP^1$ with a (profinite) tree structure
in the sense of Jonsson \cite[\S 2]{Jonsson15}.

For every $\cS\in\sP^1$, the tangent (or directions) space $T_{\cS}\sP^1$ 
of $\sP^1$ at $\cS$ is
\begin{gather*}
 T_{\cS}\sP^1:=\bigl\{\vec{v}=\widevec{\cS\cS'}:\text{the germ of a non-empty left half open interval }(\cS,\cS']\bigr\};
\end{gather*}
then
$\#T_{\cS}\sP^1=1$ if and only if $\cS\in\bP^1\cup\sH^1_{\mathrm{IV}}$, 
$\#T_{\cS}\sP^1=2$ if and only if $\cS\in\sH^1_{\mathrm{III}}$, and
$T_{\cS}\sP^1\cong\bP^1(k)$ 
if and only if $\cS\in\sH^1_{\mathrm{II}}$ (see \eqref{directionreduction}
and Facts \ref{th:Mobius}, \ref{th:Mobiustangent} below).
Identifying each $\vec{v}\in T_{\cS}\sP^1$ with
\begin{gather*}
 U_{\vec{v}}=U_{\cS,\vec{v}}:=\bigl\{\cS'\in\sP^1\setminus\{\cS\}:\widevec{\cS\cS'}=\vec{v}\bigr\}\subset 2^{\sP^1}, 
\end{gather*}
the collection $(U_{\cS,\vec{v}})_{\cS\in\sP^1,\vec{v}\in T_{\cS}\sP^1}$ is a quasi open basis of the 
(Gel'fand, weak, pointwise, or observer) topology on $\sP^1$(, and 
both $\bP^1$ and $\sH^1_{\mathrm{II}}$ are dense in $\sP^1$), 
and 
for every $\cS\in\sH^1_{\mathrm{II}}$, 
we identify $T_{\cS}\sP^1$ with $S(\{\cS\})\setminus\{\{\cS\}\}$
by the canonical bijection
\begin{gather*}
T_{\cS}\sP^1\ni\vec{v}\leftrightarrow U_{\vec{v}}\in S(\{\cS\})\setminus\{\{\cS\}\}.
\end{gather*}

The Gauss (or canonical) point $\cS_G\in\sH^1_{\mathrm{II}}$ is represented by
(the constant sequence of) the $K$-closed unit disk, that is, 
the ring $\cO_K=B(0,1)$ of $K$-integers; 
the unique maximal ideal in $\cO_K$
is $\cM_K:=\{z\in K:|z|<1\}$, and
\begin{gather*}
 k=k_K:=\cO_K/\cM_K 
\end{gather*}
is the residue field of $K$, which 
is still algebraically closed under the standing assumption on $K$. 
The residue characteristic of $K$ is $\Char k$.

The reduction $\tilde{a}\in\bP^1(k)$ of a point $a\in\bP^1(K)$ 
is defined by the point $\tilde{a_1}/\tilde{a_0}\in\bP^1(k)$, where
$a_1,a_0\in K$ are chosen so that
$a=a_1/a_0$ (regarding $1/0=\infty\in\bP^1$) and that $\max\{|a_0|,|a_1|\}=1$
(so $\tilde{\infty}=\infty\in\bP^1(k)=k\cup\{\infty\}$).
There is also a canonical bijection
\begin{gather}
 T_{\cS_G}\sP^1\ni\widevec{\cS_Ga}\leftrightarrow\tilde{a}\in\bP^1(k).
\label{directionreduction} 
\end{gather}

For more details on (dynamics on) $\sP^1$, 
see e.g.\ the books \cite{BR10,BenedettoBook} and the survey article \cite{Jonsson15}.

\subsection{Dynamics on $\sP^1$ and their reductions}
For every $h\in K(z)$, writing 
\begin{gather*}
 h(z)=\frac{P(z)}{Q(z)},\quad
P(z)=\sum_{j=0}^{\deg h}a_j z^j\in K[z],\quad\text{and}\quad
Q(z)=\sum_{\ell=0}^{\deg h}b_\ell z^\ell\in K[z],
\end{gather*}
this $h$ is regarded as the point 
$[b_0:\cdots:b_{\deg h}:a_0:\cdots:a_{\deg h}]\in\bP^{2(\deg h)+1}(K)$. 
Then choosing $P,Q$
so that 
\begin{gather*}
 \max\bigl\{|b_0|,\ldots,|b_{\deg h}|,|a_0|,\ldots,|a_{\deg h}|\bigr\}=1,
\end{gather*}
we obtain the point
$\tilde{h}=\bigl[\widetilde{b_0}:\cdots:\widetilde{b_{\deg h}}:\widetilde{a_0}:\cdots:\widetilde{a_{\deg h}}\bigr]
\in\bP^{2(\deg h)+1}(k)$; this point $\tilde{h}\in\bP^{2(\deg h)+1}(k)$ is formally written as
\begin{gather*}
\tilde{h}=H_{\tilde{h}}\phi_{\tilde{h}},
\end{gather*}
where we set
$\tilde{P}(\zeta):=\sum_{j=0}^{\deg h}\widetilde{a_j}\zeta^j\in k[\zeta]$,
$\tilde{Q}(z):=\sum_{\ell=0}^{\deg h}\widetilde{b_\ell}\zeta^\ell\in k[\zeta]$,
\begin{gather*}
 H_{\tilde h}(X_0,X_1):=\GCD\bigl(X_0^{\deg h}\tilde{Q}(X_1/X_0),X_0^{\deg h}\tilde{P}(X_1/X_0)\bigr)\in\bigcup_{\ell=0}^{\deg h}k[X_0,X_1]_\ell\setminus\{0\},\\
\text{and}\quad 
\phi_{\tilde{h}}(\zeta)
:=\frac{\tilde{P}(\zeta)/H_{\tilde{h}}(1,\zeta)}{\tilde{Q}(\zeta)/H_{\tilde{h}}(1,\zeta)}\in k(\zeta)
\end{gather*}
($H_{\tilde{h}}$ is unique up to multiplication in $k^*$).
The rational function $\phi_{\tilde{h}}\in k(\zeta)$ on $\bP^1(k)$
is called the reduction of $h$, the degree of which equals
$\deg h-\deg H_{\tilde{h}}$. 

\begin{notation}
 When $\deg H_{\tilde{h}}>0$, we denote by
 $[H_{\tilde{h}}=0]$ the effective divisor on $\bP^1(k)$
 defined by the zeros of $H_{\tilde{h}}$ on $\bP^1(k)$
 taking into account their multiplicities, so that
 $\deg[H_{\tilde{h}}=0]=\deg H_{\tilde{h}}$. When $\deg H_{\tilde{h}}=0$, we set
 $[H_{\tilde{h}}=0]:=0$ on $\bP^1(k)$ by convention.
\end{notation}

The action on $\bP^1$ of $h\in K(z)$ extends continuously to 
that on $\sP^1$, and if in addition $\deg h>0$, then 
this extended action is surjective, open, and
fiber-discrete, and preserves $\bP^1,\sH^1_{\mathrm{II}},\sH^1_{\mathrm{III}}$,
and $\sH^1_{\mathrm{IV}}$, as already mentioned in Section \ref{sec:intro}.
Then
\begin{gather}
 h^{-1}(\cS_G)=\{\cS_G\}\quad\Leftrightarrow\quad\tilde{h}=\phi_{\tilde{h}}\quad\Leftrightarrow\quad\deg H_{\tilde{h}}=0.\label{eq:goodred}
\end{gather}

\begin{fact}[Rivera-Letelier {\cite{Juan03}, and see also \cite[Corollary 9.27]{BR10}}]\label{th:constantreduction}
$\deg(\phi_{\tilde{h}})>0$ if and only if $h(\cS_G)=\cS_G$. Moreover,
\begin{gather}
 \phi_{\tilde{h}}\equiv\tilde{z}\text{ for some }z\in\bP^1
\quad\Rightarrow\quad 
 \widevec{\cS_Gh(\cS_G)}=\widevec{\cS_Gz}.
\label{eq:constantdirection}
\end{gather}
\end{fact} 

\begin{fact}\label{th:Mobius}
 The group $\PGL(2,K)$ of M\"obius transformations on $\bP^1$
 acts transitively on $\sH^1_{\mathrm{II}}$, and
 $\PGL(2,\cO_K)$ is the stabilizer subgroup of $\cS_G$ in $\PGL(2,K)$.
\end{fact}

From now on, suppose that $\deg h>0$.

\subsection{The tangent maps and the directional/surplus local degrees of rational functions}
For the details on this and the next subsections, 
see Rivera-Letelier \cite{Juan05,Juan03}; see also Jonsson \cite[\S4.5]{Jonsson15} for an algebraic treatment. 

For every $\cS\in\sP^1$, 
the tangent map $h_*=(h_*)_{\cS}:T_{\cS}\sP^1\to T_{h(\cS)}\sP^1$ 
of $h$ at $\cS$ is defined so that 
for every $\vec{v}=\widevec{\cS\cS'}\in T_{\cS}\sP^1$, if $\cS'$ is close enough to $\cS$, then
$h$ maps the interval $[\cS,\cS']$ onto the interval $[h(\cS),h(\cS')]$ 
homeomorphically, and
\begin{gather*}
 h_*(\vec{v})=\widevec{h(\cS)h(\cS')}.
\end{gather*}
Moreover, 
for every $\cS\in\sH^1_{\mathrm{II}}$ and every $\vec{v}\in T_{\cS}\sP^1$, 
there is the directional local degree $m_{\vec{v}}(h)\in\bN$ 
(indeed $\in\{1,\ldots,\deg_{\cS}(h)\}$) of $h$ 
on $U_{\vec{v}}$ 
such that choosing any $A,B\in\PGL(2,K)$ satisfying 
$B^{-1}(\cS)=A(h(\cS))=\cS_G$ (so $\deg(\widetilde{A\circ h\circ B})>0$
by Fact \ref{th:constantreduction})
and writing $(B^{-1})_*(\vec{v})=\widevec{\cS_Gz}$ and
$A_*(h_*(\vec{v}))=\widevec{\cS_Gw}$
by some $z,w\in\bP^1$, we have
\begin{gather}
 \phi_{\widetilde{A\circ h\circ B}}(\tilde{z})=\tilde{w}\quad\text{and}\label{eq:tangentreduct} \\
 m_{\vec{v}}(h)
=\deg_{\tilde{z}}\bigl(\phi_{\widetilde{A\circ h\circ B}}\bigr).\label{directdegdef}
\end{gather}
For every $\cS\in\sP^1\setminus\sH^1_{\mathrm{II}}$
and every $\vec{v}\in T_{\cS}\sP^1$, we set $m_{\vec{v}}(h):=\deg_{\cS}(h)$.

\begin{fact}[decomposition of the local degree {\cite[Proposition 3.5]{Juan05}}]\label{th:directionalpullback}
For every $\cS\in\sP^1$, also using the notation in the above paragraph
if $\cS\in\sH^1_{\mathrm{II}}$,
we have
\begin{gather}
(1\le)\deg_{\cS}(h)
=\sum_{\vec{v}\in T_{\cS}\sP^1:h_*(\vec{v})=\vec{w}}m_{\vec{v}}(h)\Bigl(=\deg\bigl(\phi_{\widetilde{A\circ h\circ B}}\bigr)\text{ if }\cS\in\sH^1_{\mathrm{II}}\Bigr)\quad
\text{for any }\vec{w}\in T_{h(\cS)}\sP^1;
\label{eq:totallocaldegree}
\end{gather}
in particular, $h_*:T_{\cS}\sP^1\to T_{h(\cS)}\sP^1$ is surjective.
\end{fact}

\begin{fact}[a non-archimedean argument principle {\cite[Lemma 2.1]{Juan03}}]\label{th:surplus}
For every $\cS\in\sP^1$ and every $\vec{v}\in T_{\cS}\sP^1$,
there is the surplus local degree 
$s_{\vec{v}}(h)\in
\{0,1,\ldots,\deg_{\cS}(h)\}$ 
of $h$ on $U_{\vec{v}}$ such that
for every $\cS'\in\sP^1\setminus\{h(\cS)\}$,
\begin{gather}
 (h^*\delta_{\cS'})(U_{\vec{v}})=
\begin{cases}
 m_{\vec{v}}(h)+s_{\vec{v}}(h) 
&\text{if }U_{h_*(\vec{v})}\ni \cS',\\
s_{\vec{v}}(h) & \text{otherwise};
\end{cases}\label{eq:argumentnonarchi}
\end{gather}
moreover, $h(U_{\vec{v}})$ is either $\sP^1$ 
or $U_{h_*(\vec{v})}$, the latter possibility in which
is the case if and only if $s_{\vec{v}}(h)=0$. 
For every $\cS\in\sP^1$, 
$s_{\vec{v}}(h)>0$ for at most finitely many $\vec{v}\in T_{\cS}\sP^1$, and then 
\begin{gather}
 \sum_{\vec{v}\in T_{\cS}\sP^1}s_{\vec{v}}(h)=
\deg h-\deg_{\cS}(h)\label{eq:totalsurplus}
\end{gather}
since fixing any $\cS'\in\sP^1\setminus\{h(\cS)\}$, we have 
\begin{multline*}
 \deg h=(h^*\delta_{\cS'})(\sP^1)=(h^*\delta_{\cS'})(\sP^1\setminus\{\cS\})\\
=\sum_{\vec{v}\in T_{\cS}\sP^1:h_*(\vec{v})=\widevec{\cS\cS'}}m_{\vec{v}}(h)+
\sum_{\vec{v}\in T_{\cS}\sP^1}s_{\vec{v}}(h)
=\deg_{\cS}(h)+
\sum_{\vec{v}\in T_{\cS}\sP^1}s_{\vec{v}}(h).
\end{multline*}
\end{fact}

\begin{fact}\label{th:Mobiustangent}
In the case that $h\in\PGL(2,K)$, the tangent map
$h_*:T_{\cS}\sP^1\to T_{h(\cS)}\sP^1$ is bijective, and
for every $\cS\in\sP^1$ and every $\vec{v}\in T_{\cS}\sP^1$,
$h(U_{\vec{v}})=U_{h_*(\vec{v})}$.
\end{fact}

\begin{fact}[Faber {\cite[Lemma 3.17]{Faber13topologyI}}]\label{th:surplusalgclosed}
For every $\cS\in\sH^1_{\mathrm{II}}$ and
every $\vec{v}\in T_{\cS}\sP^1$,
choosing any such $A,B\in\PGL(2,K)$ that $B^{-1}(\cS)=A(h(\cS))=\cS_G$
and any such $z\in\bP^1$ that
$(B^{-1})_*(\vec{v})=\widevec{\cS_Gz}$ 
(as in the paragraph before Fact \ref{th:directionalpullback}), we have
\begin{gather}
 s_{\vec{v}}(h)
\begin{cases}
 =
\ord_{\zeta=\tilde{z}}\bigl[H_{\widetilde{A\circ h\circ B}}=0\bigr] &\text{if }\deg H_{\widetilde{A\circ h\circ B}}>0,\\
\equiv 0 & \text{otherwise}.
\end{cases}\label{eq:surplusfaber}
\end{gather} 
\end{fact}

\subsection{The hyperbolic metric $\rho$ on $\sH^1$ and the piecewise affine action of $h$ on $(\sH^1,\rho)$}
The hyperbolic metric $\rho$ on $\sH^1$,
which is defined so that
\begin{gather*}
 \rho\bigl(\cS_1,\cS_2\bigr)=\log\biggl(\frac{\diam_{|\cdot|}\cS_2}{\diam_{|\cdot|}\cS_1}\biggr)\quad\text{ if }\cS_1\preceq\cS_2,
\end{gather*}
would be used at some part in the proof of Theorem \ref{th:computation}.
The topology on $(\sH^1,\rho)$ is finer than the relative topology on $\sH^1$
from $\sP^1$.

\begin{fact}[{\cite[Proposition 3.5]{Juan05}}]
 For every $\cS\in\sP^1$ and every 
 $\vec{v}=\widevec{\cS\cS'}\in T_{\cS}\sP^1$,
 if $\cS'$ is close enough to $\cS$, then for every $\cS''\in(\cS,\cS']$,
 \begin{gather}
 \rho\bigl(h(\cS''),h(\cS')\bigr)
 =m_{\vec{v}}(h)\cdot\rho(\cS'',\cS'),\label{eq:Lipschitz}
 \end{gather}
 which still holds for $\cS''\in[\cS,\cS']$ if $\cS\in\sH^1$.
\end{fact}

\subsection{Quantized local degrees and quantized pullbacks}\label{sec:quantized}
Let us precisely define the quantized local degrees $m_{V,U}(h)$, 
mentioned in Subsection \ref{sec:skeltal}, 
in terms of the (directional/surplus) local degrees of $h$, 
and then also (re-)define the quantized pullback operator 
$h_G^*:M(\Gamma_h)\to M(\Gamma_G)$. 
Recall
\begin{gather*}
 \Gamma_G:=\{\cS_G\}\quad\text{and}\quad \Gamma_h:=\{\cS_G,h(\cS_G)\}
\quad\text{in }\sH^1_{\mathrm{II}}.
\end{gather*}

\begin{definition}[the quantized local degree]\label{th:multfactor}
 For every $U_{\vec{v}}\in S(\Gamma_G)\setminus\{\{\cS_G\}\}=T_{\cS_G}\sP^1$ 
and every $V\in S(\Gamma_h)$, 
set
\begin{align*}
  m_{V,U_{\vec{v}}}(h):=&
 \begin{cases}
 m_{\vec{v}}(h)+s_{\vec{v}}(h) &\text{if }V\subset U_{h_*(\vec{v})},\\
 s_{\vec{v}}(h) &\text{if }V\cap U_{h_*(\vec{v})}=\emptyset
 \end{cases}\\
 \underset{\eqref{eq:argumentnonarchi}}{=}&(h^*\delta_{\cS'})(U_{\vec{v}})\text{ for any }\cS'\in V
\text{ if }V\in S(\Gamma_h)\setminus\{\{h(\cS_G)\}\},
\end{align*}
 and for every $V\in S(\Gamma_h)$, set 
 \begin{align*}
 m_{V,\{\cS_G\}}(h)
 :=&\begin{cases}
   \deg_{\cS_G}(h) & \text{if }V=\{h(\cS_G)\},\\
   0 & \text{if }V\in S(\Gamma_h)\setminus\{\{h(\cS_G)\}\}
  \end{cases}\\
  \underset{\eqref{eq:pullbackdirac}}{=}&(h^*\delta_{\cS'})(\{\cS_G\})\quad\text{for any }\cS'\in V.
 \end{align*}
\end{definition}
\begin{fact}
The fundamental equality
 \begin{gather}
 \sum_{U\in S(\Gamma_G)}m_{V,U}(h)=\deg h
\quad\text{for any }V\in S(\Gamma_h) 
\label{eq:totalfactor}
 \end{gather}
holds; indeed, for every $V\in S(\Gamma_h)\setminus\{\{h(\cS_G)\}\}$, 
there is a unique $\vec{w}\in T_{h(\cS_G)}\sP^1$
satisfying $V\subset U_{\vec{w}}$, and then
 \begin{multline*}
 \sum_{U\in S(\Gamma_G)}m_{V,U}(h)
 =\sum_{\vec{v}\in T_{\cS_G}\sP^1:h_*(\vec{v})=\vec{w}}m_{\vec{v}}(h)
 +\sum_{\vec{v}\in T_{\cS_G}\sP^1}s_{\vec{v}}(h)+0\\
 \underset{\eqref{eq:totallocaldegree}\&\eqref{eq:totalsurplus}}{=}\deg_{\cS_G}(h)+\bigl(\deg h-\deg_{\cS_G}(h)\bigr)
 =\deg h,
 \end{multline*}
 and similarly, 
 \begin{multline*}
 \sum_{U\in S(\Gamma_G)}m_{\{h(\cS_G)\},U}(h)
 =\sum_{\vec{v}\in T_{\cS_G}\sP^1}s_{\vec{v}}(h)+\deg_{\cS_G}(h)\\
 \underset{\eqref{eq:totalsurplus}}{=}\bigl(\deg h-\deg_{\cS_G}(h)\bigr)+\deg_{\cS_G}(h)=\deg h.
 \end{multline*}
\end{fact}

 The quantized pushforward operator $h_{G,*}$ from the space of
 measurable functions on $\sP^1/S(\Gamma_G)$ to
 that of measurable functions on $\sP^1/S(\Gamma_h)$ 
 is defined so that for every measurable function
 $\psi$ on $\sP^1/S(\Gamma_G)$,
 the measurable function $h_{G,*}\psi$ on $\sP^1/S(\Gamma_h)$
 satisfies
\begin{gather*}
 (h_{G,*}\psi)(V)=\sum_{U\in S(\Gamma_h)}m_{V,U}(h)\psi(U)
\quad\text{for any }V\in S(\Gamma_h)\quad\text{or equivalently}\\
 (\pi_{\Gamma_h})^*(h_{G,*}\psi)
  \equiv\sum_{U\in S(\Gamma_G)}m_{V,U}(h)\cdot((\pi_{\Gamma_G})^*\psi)|U
  \quad\text{on each }V\in S(\Gamma_h),
\end{gather*}
so in particular
\begin{gather}
(\pi_{\Gamma_h})^*(h_{G,*}\psi) =\sum_{\vec{v}\in T_{\cS_G}\sP^1}(h^*\delta_{\,\cdot})(U_{\vec{v}})\cdot((\pi_{\Gamma_G})^*\psi)|U_{\vec{v}}
 \quad\text{on }\sP^1\setminus\{h(\cS_G)\}.\label{eq:pushfactor}
\end{gather}
 The quantized pullback operator $h_G^*:M(\Gamma_h)\to M(\Gamma_G)$ 
 is the transpose of this quantized pushforward operator $h_{G,*}$
 so in particular that
 for every $\omega\in M(\Gamma_h)$,
 the measure $h_G^*\omega\in M(\Gamma_G)$ satisfies
 \begin{align}
\notag \bigl(h_G^*\omega\bigr)(\{U\}) 
  =&\left\langle 1_{\{U\}},h_G^*\omega\right\rangle
  =\left\langle h_{G,*}(1_{\{U\}}),\omega\right\rangle\\
\notag =&\int_{\sP^1/S(\Gamma_h)}\Bigl(\sum_{W\in S(\Gamma_G)}m_{V,W}(h)
  \cdot 1_{\{U\}}(W)\Bigr)\omega(V)\\
 =&\int_{\sP^1/S(\Gamma_h)}m_{V,U}(h)\omega(V)
\quad\text{for any }U\in S(\Gamma_G).\label{eq:pullbackdefining}
 \end{align}

\section{Degenerating balanced property for degenerating 
weak limit points of the maximal entropy measures on $\bP^1(\bC)$}
\label{sec:quantizedbalanced}

We follow the presentation in \cite[\S 2.1-\S2.4]{DF14}.

Fixing $r\in(0,1)$ (e.g.\ $r=e^{-1}$) once and for all, 
the field $\bC((t))$ of Laurent series around $t=0$ over $\bC$
is equipped with the non-trivial and non-archimedean absolute value
\begin{gather}
 |x|_r=r^{\min\{n\in\bZ:\,a_n\neq 0\}}\label{eq:normLaurent}
\end{gather}
for $x(t)=\sum_{n\in\bZ}a_nt^n\in\bC((t))$
(under the convention that $\min\emptyset=+\infty$ and $r^{+\infty}=0$),
which restricts to the trivial absolute value on $\bC$.

An algebraic closure $\overline{\bC((t))}$ of $\bC((t))$ 
is the field of Puiseux series around $t=0$ over $\bC$, 
$|\cdot|_r$ extends to $\overline{\bC((t))}$ as an absolute value,
and 
the completion $\bL$ of $\overline{\bC((t))}$
is the field of formal Puiseux series around $t=0$ over $\bC$ and 
is still algebraically closed. We note that $\cO(\bD)[t^{-1}]\subset\bC((t))$,
\begin{gather*}
\bC\subset\cO(\bD)\subset\cO_{\bC((t))}
=\biggl\{\sum_{n\in\bZ}a_nt^n\in\bC((t)):a_n=0\text{ if }n<0\biggr\}=\bC[[t]],\\
 \cM_{\bC((t))}=t\cdot\cO_{\bC((t))},\\
 k_{\bL}=k_{\bC((t))}=\bC\text{ (as fields),}\quad\text{and}\\
 T_{\cS_G}\sP^1(\bL)\cong\bP^1(k_{\bL})=\bP^1(\bC)\quad
 (\text{the bijection
 is the canonical one in \eqref{directionreduction}}).
\end{gather*}

\begin{notation}
 Let $M(\bP^1(\bC))$ be
 the space of all complex Radon measures on $\bP^1(\bC)=\bC\cup\{\infty\}$.
 The pullback 
of each $\mu\in M(\bP^1(\bC))$ under a
 rational function $R\in\bC(z)$ on $\bP^1(\bC)$ of degree $>0$ is 
$R^*\mu:=\int_{\bP^1(\bC)}(\sum_{w\in R^{-1}(z)}(\deg_wR)\delta_w)\mu(z)$ 
on $\bP^1(\bC)$, 
where for each $z\in\bP^1(\bC)$, $\delta_z$ is the Dirac measure at $z$ on $\bP^1(\bC)$; if $R$ is constant, then $R^*\mu:=0$ by convention. Also set
\begin{align*}
M^1(\bP^1(\bC)):=&\bigl\{\mu\in M(\bP^1(\bC)):\mu\ge 0\text{ and }\mu(\bP^1(\bC))=1\bigr\}\quad\text{and}\\
M^1(\bP^1(\bC))^{\dag}:=&\bigl\{\mu\in M^1(\bP^1(\bC)):\mu\text{ is purely atomic}\}.
\end{align*}
\end{notation}

\begin{fact}[the maximal entropy measure on $\bP^1(\bC)$
{\cite{Brolin,Lyubich83,FLM83}}]\label{maxentropy}
For a rational function $R\in\bC(z)$ on $\bP^1(\bC)$ of degree $>1$,
the equilibrium (or canonical, and indeed the unique maximal entropy)
measure $\mu_R$ of $R$ on $\bP^1(\bC)$ is the unique
$\mu\in M^1(\bP^1(\bC))$ satisfying $R^*\mu=(\deg R)\mu$
on $\bP^1(\bC)$ and $\mu(E(R))=0$, where $E(R):=\{a\in\bP^1(\bC):\#\bigcup_{n\in\bN}R^{-n}(a)<\infty\}$.
Then for every $n\in\bN$, $\mu_{R^n}=\mu_R$ on $\bP^1(\bC)$ and $E(R^n)=E(R)$. The measure $\mu_R$ is $\PGL(2,\bC)$-equivariant in that for every M\"obius transformation $M\in\PGL(2,\bC)$ on $\bP^1(\bC)$, $\mu_{M\circ R\circ M^{-1}}=M_*\mu_R$ on $\bP^1(\bC)$.

When $R\in\bC[z]$ or equivalently $R(\infty)=\infty\in E(R)$, 
$\mu_R$ is supported by $\partial(K_R)$,
 where the filled-in Julia set $K_R:=\{z\in\bC:\limsup_{n\to\infty}|R^n(z)|<+\infty\}$ of $R$ is a compact subset in $\bC$.
\end{fact}

Let $h\in(\cO(\bD)[t^{-1}])(z)(\subset\bL(z))$ 
be a meromorphic family 
of rational functions on $\bP^1(\bC)$, and let us regard
$\tilde{h}=H_{\tilde{h}}\phi_{\tilde{h}}\in\bP^{2(\deg h)+1}(k_{\bL})$ 
as a point in $\bP^{2(\deg h)+1}(\bC)$,
$\phi_{\tilde{h}}$ as a rational function on $\bP^1(\bC)$
of degree $\deg h-\deg H_{\tilde{h}}$, 
and the effective divisor $[H_{\tilde{h}}=0]$ on $\bP^1(k_{\bL})$
as that on $\bP^1(\bC)$ 
and in turn also as the Radon measure $\sum_{z\in\bP^1(\bC)}(\ord_z[H_{\tilde{h}}=0])\delta_z$ on $\bP^1(\bC)$, under $k_{\bL}=\bC$ as fields. Then
\begin{gather}
\lim_{t\to 0}h_t=\phi_{\tilde{h}}
\quad\text{locally uniformly on }\bP^1(\bC)\setminus\bigl(\supp[H_{\tilde{h}}=0]\bigr).\label{eq:locunifoutside} 
\end{gather}

\begin{definition}
 For every $\mu\in M^1(\bP^1(\bC))$,
 the (possibly degenerating) 
 pullback $\tilde{h}^*\mu\in M(\bP^1(\bC))$
 of $\mu$ under $\tilde{h}$ is defined by
 \begin{gather}
 \tilde{h}^*\mu
 :=(\phi_{\tilde{h}})^*\mu+[H_{\tilde{h}}=0]\quad\text{on }\bP^1(\bC),\label{eq:degenpullback}
 \end{gather} 
 still satisfying $\bigl(\tilde{h}^*\mu\bigr)(\bP^1(\bC))=\deg h$.
\end{definition}

Recall Fact \ref{th:constantreduction}.
The following target rescaling theorem is a special case of
\cite[Lemma 3.7]{Kiwi15} (see also {\cite[Lemma 2.1]{DF14}}).

\begin{theorem}
\label{th:postfamily}
For every meromorphic family $f\in(\cO(\bD)[t^{-1}])(z)(\subset\bL(z))$ of rational functions
on $\bP^1(\bC)$ of degree $>1$, 
there is a meromorphic family $A\in(\cO(\bD)[t^{-1}])(z)$ 
of M\"obius transformations on $\bP^1(\bC)$
such that $(A\circ f)(\cS_G)=\cS_G$ in $\sP^1(\bL)$.
Such 
a family
$A$ is unique up to a postcomposition to $A$ of any meromorphic family
$B\in(\cO(\bD)[t^{-1}])(z)$ of M\"obius transformations on $\bP^1(\bC)$
satisfying $\tilde{B}=\phi_{\tilde{B}}\in\PGL(2,\bC)$.
\end{theorem}

Recall also \eqref{eq:goodred}.
The degenerating $f$-balanced property of the pair $\mu$ (the former half in
\eqref{eq:reductionconstant})
is a consequence of \eqref{eq:locunifoutside}
and the complex argument principle. The proof of the
purely atomicness of $\mu$ (the latter half in 
\eqref{eq:reductionconstant}) is more involved.

\begin{theorem}[a consequence of {\cite[Theorems 2.4 and A]{DF14}}]\label{th:balancedgeneral}
Let 
\begin{gather*}
 f\in\bigl(\cO(\bD)[t^{-1}]\bigr)(z)(\subset\bL(z)) 
\end{gather*}
be a meromorphic family 
of rational functions on $\bP^1(\bC)$ of degree $d>1$
satisfying $f^{-1}(\cS_G)\neq\{\cS_G\}$ in $\sP^1(\bL)$, 
let $A\in(\cO(\bD)[t^{-1}])(z)$ be
a meromorphic family  of M\"obius transformations on $\bP^1(\bC)$
such that $(A\circ f)(\cS_G)=\cS_G$, and
let 
\begin{gather*}
 \mu_C=\lim_{j\to\infty}\mu_{f_{t_j}},\quad
\mu_E=\lim_{j\to\infty}(A_{t_j})_*\mu_{f_{t_j}}\in M^1(\bP^1(\bC)) 
\end{gather*}
be weak limit points on $\bP^1(\bC)$ as $t\to 0$ of 
the families
$(\mu_{f_t})_{t\in\bD^*}$ and $((A_t)_*\mu_{f_t})_{t\in\bD^*}$ 
of the unique maximal entropy measures $\mu_{f_t}$
and $(A_t)_*\mu_{f_t}=\mu_{A_t\circ f_t\circ A_t^{-1}}$ on $\bP^1(\bC)$ of $f_t$ and $A_t\circ f_t\circ A_t^{-1}$,
respectively, for some sequence
$(t=t_j)$ in $\bD^*$ tending to $0$ as $j\to\infty$. 
Then
\begin{gather}
(\widetilde{A\circ f})^*\mu_E=d\cdot\mu_C\quad\text{on }\bP^1(\bC)
\quad\text{and}\quad
\mu:=(\mu_C,\mu_E)\in\bigl(M^1(\bP^1(\bC))^\dag\bigr)^2.
\label{eq:reductionconstant}
\end{gather}
\end{theorem}

\section{A direct translation}\label{sec:direct}

Pick a meromorphic family $f\in(\cO(\bD)[t^{-1}])(z)(\subset\bL(z))$
of rational functions on $\bP^1(\bC)$ of degree $d>1$, and
suppose that $f^{-1}(\cS_G)\neq\{\cS_G\}$ in $\sP^1(\bL)$.
Choose a meromorphic family $A\in(\cO(\bD)[t^{-1}])(z)$ of 
M\"obius transformations on $\bP^1(\bC)$ 
such that $(A\circ f)(\cS_G)=\cS_G$ (by Theorem \ref{th:postfamily}).
Recall also
\begin{gather*}
 \Gamma_G:=\{\cS_G\}\quad\text{and}\quad\Gamma_f:=\{\cS_G,f(\cS_G)\}
\quad\text{in }\sH^1_{\mathrm{II}}(\bL). 
\end{gather*}
From Fact \ref{th:constantreduction} and \eqref{eq:goodred}, 
the following 
five
statements
\begin{gather*}
\Gamma_G=\Gamma_f,\quad
f(\cS_G)=\cS_G,\quad
\deg(\phi_{\tilde{f}})>0,\quad\text{and moreover}\\
A(\cS_G)=\cS_G\quad\text{and}\\
\tilde{A}=\phi_{\tilde{A}}\in\PGL(2,k_{\bL})=\PGL(2,\bC)\,\, (\text{under }k_{\bL}=\bC\text{ as fields, here and below})
\end{gather*}
are equivalent.
Alternatively when $\Gamma_G\neq\Gamma_f$, 
there are 
$h_A,a_A\in\bP^1(\bC)$ such that
\begin{multline}
 \supp[H_{\tilde{A}}=0]=\{h_A\}\quad\text{in }\bP^1(\bC),\quad
\phi_{\tilde{A}}\equiv a_A\quad\text{on }\bP^1(\bC),\\
\text{and moreover }\phi_{\tilde{f}}\equiv h_A\quad\text{on }\bP^1(\bC)\,\,
\bigl(\text{by }\eqref{eq:locunifoutside}\text{ and Fact }\ref{th:constantreduction}\bigr).\label{eq:redconst}
\end{multline}  

We note that
$T_{f(\cS_G)}\sP^1(\bL)
\underset{(A^{-1})_*}{\overset{\cong}{\leftarrow}} T_{\cS_G}\sP^1(\bL)
\underset{\eqref{directionreduction}}{\cong}\bP^1(k_{\bL})=\bP^1(\bC)$,
also recalling Fact \ref{th:Mobiustangent}.

\begin{lemma}
 When $\Gamma_f\neq\Gamma_G$, we have
\begin{gather}
  (A^{-1})_*\bigl(\widevec{\cS_GA(\cS_G)}\bigr)
 =\widevec{f(\cS_G)\cS_G}.\label{eq:consistent} 
\end{gather}
\end{lemma} 

\begin{proof}
If $(A^{-1})_*(\vec{v})=\widevec{f(\cS_G)\cS_G}
\bigl(=\widevec{A^{-1}(\cS_G)\cS_G}\bigr)$ for some (indeed unique)
$\vec{v}\in T_{\cS_G}\sP^1(\bL)$, then we have
$\cS_G\in U_{(A^{-1})_*(\vec{v})}$, which yields
$A(\cS_G)\in A(U_{(A^{-1})_*(\vec{v})})
=U_{A_*(A^{-1})_*(\vec{v})}=U_{\vec{v}}$
(using Fact \ref{th:Mobiustangent}),
and in turn $\vec{v}=\widevec{\cS_GA(\cS_G)}$.
\end{proof}

\begin{lemma}\label{th:annulus}
When $\Gamma_f\neq\Gamma_G$, for any $\tilde{x},\tilde{y}
\in\bP^1(k_{\bL})=\bP^1(\bC)$ 
$($and any representatives $x,y\in\bP^1(\bL)$ of $\tilde{x},\tilde{y}$,
respectively$)$, we have
 \begin{gather}
 \begin{cases}
 \widevec{\cS_Gx}=\widevec{\cS_Gf(\cS_G)}\quad\text{in }T_{\cS_G}\sP^1(\bL)
  &\Leftrightarrow\quad\tilde{x}=h_A\quad\text{in }\bP^1(k_{\bL})=\bP^1(\bC),\\
  (A^{-1})_*(\widevec{\cS_Gy})
 =\widevec{f(\cS_G)\cS_G}\quad\text{in }T_{f(\cS_G)}\sP^1(\bL)
 &\Leftrightarrow\quad\tilde{y}=a_A\quad\text{in }\bP^1(k_{\bL})=\bP^1(\bC).
 \end{cases}\label{eq:intersection}
 \end{gather}
\end{lemma}
\begin{proof}
 The former assertion is by $\phi_{\tilde{f}}\equiv h_A$ on $\bP^1(\bC)$
 (in \eqref{eq:redconst}) and \eqref{eq:constantdirection}.
 On the other hand, by \eqref{eq:consistent}, we have
\begin{gather*}
 (A^{-1})_*(\widevec{\cS_Gy})=\widevec{f(\cS_G)\cS_G}
 \quad\Leftrightarrow\quad
 \widevec{\cS_Gy}\bigl(=A_*(\widevec{f(\cS_G)\cS_G})\bigr)=\widevec{\cS_GA(\cS_G)},
\end{gather*} 
so the latter assertion holds by $\phi_{\tilde{A}}\equiv a_A$ on $\bP^1(\bC)$
(in \eqref{eq:redconst}) and \eqref{eq:constantdirection}.
\end{proof}

\begin{definition}[the admissibility of $\mu$ and the construction of 
the measure $\omega_{\mu}$]\label{th:direct}
For every $\mu=(\mu_C,\mu_E)\in (M^1(\bP^1(\bC)))^2$
satisfying the following admissibility
\begin{gather}
 \begin{cases}
  \tilde{A}^*\mu_E=\mu_C\quad\text{on }\bP^1(\bC)
 &\text{when }\Gamma_f=\Gamma_G\bigl(\Leftrightarrow\tilde{A}=\phi_{\tilde{A}}\Leftrightarrow A(\cS_G)=\cS_G\bigr),\\
 \mu_C(\{h_A\})+\mu_E(\{a_A\})\ge 1
 &\text{when }\Gamma_f\neq\Gamma_G
 \end{cases} \label{eq:compatibility}
\end{gather} 
(for $A$), there is a unique probability measure
\begin{gather*}
 \omega_{\mu}\in M^1(\Gamma_f)\quad\bigl(\text{and indeed }
\omega_{\mu}\in M^1(\Gamma_f)^\dag\text{ if }\mu\in (M^1(\bP^1(\bC))^\dag)^2\bigr)
\end{gather*}
on $\sP^1/S(\Gamma_f)=S(\Gamma_f)$
such that, writing $\mu_C=\nu_C+\tilde{\nu}_C$ (resp.\ $\mu_E=\nu_E+\tilde{\nu}_E$)
in $M(\bP^1)$ where 
$\nu_C$ (resp.\ $\nu_E$) has no atoms on $\bP^1(\bC)$ and
$\tilde{\nu}_C=\mu_C-\nu_C$ (resp.\ $\tilde{\nu}_E=\mu_E-\nu_E$) is purely atomic,
when $\Gamma_f=\Gamma_G$,
\begin{gather*}
 \begin{cases}
 \omega_{\mu}(\{\{\cS_G\}\})=\nu_E(\bP^1(\bC))\bigl(=\nu_C(\bP^1(\bC))\bigr)\quad\text{and} \\
 \omega_{\mu}\bigl(\bigl\{U_{(A^{-1})_*(\widevec{\cS_Gy})}\bigr\}\bigr)
 =\mu_E(\{\tilde{y}\})
\quad\text{for every }\tilde{y}\in\bP^1(k_{\bL})=\bP^1(\bC)\\
\hspace*{50pt}\biggl(\underset{\eqref{eq:compatibility}
\&\eqref{eq:tangentreduct}}{\Leftrightarrow} 
 \omega_{\mu}(\{U_{\widevec{\cS_Gy}}\})=\mu_C(\{\tilde{y}\})\quad\text{for every }\tilde{y}\in\bP^1(k_{\bL})=\bP^1(\bC)\biggr)
 \end{cases}
\end{gather*}
 and, when $\Gamma_f\neq\Gamma_G$(, noting also Lemma \ref{th:annulus}),
\begin{gather}
 \begin{cases}
 \omega_{\mu}(\{\{\cS_G\}\})=\nu_C(\bP^1(\bC)),\\
 \omega_{\mu}(\{U_{\widevec{\cS_Gx}}\})
 =\mu_C(\{\tilde{x}\})\quad \text{for every }\tilde{x}\in\bP^1(\bC)\setminus\{h_A\},\\
 \omega_{\mu}(\{\{f(\cS_G)\}\})=\nu_E(\bP^1(\bC)),\\
 \omega_{\mu}\bigl(\bigl\{U_{(A^{-1})_*(\widevec{\cS_Gy})}\bigr\}\bigr)
 =\mu_E(\{\tilde{y}\})\quad \text{for every }\tilde{y}\in\bP^1(\bC)\setminus\{a_A\},\quad\text{and}\\
 \omega_{\mu}\bigl(\bigl\{U_{\widevec{\cS_Gf(\cS_G)}}\cap U_{\widevec{f(\cS_G)\cS_G}}\bigr\}\bigr)
=\mu_C(\{h_A\})+\mu_E(\{a_A\})-1(\underset{\eqref{eq:compatibility}}{\ge} 0).
 \end{cases}\label{eq:omega}
\end{gather}
\end{definition}

For every $\mu=(\mu_C,\mu_E)\in (M^1(\bP^1(\bC)))^2$
satisfying the admissibility \eqref{eq:compatibility} (for $A$), we note that
\begin{gather*}
 \omega_{\mu}\bigl(S(\Gamma_f)\setminus F\bigr)=0
\quad\text{for some countable subset }F\text{ in }S(\Gamma_f), 
\end{gather*}
and also have
\begin{align}
\notag\omega_\mu\in M^1(\Gamma_f)^\dag &\Rightarrow
 (\pi_{\Gamma_f,\Gamma_G})_*\omega_\mu\in M^1(\Gamma_G)^\dag\\
 &\Rightarrow
 \mu_C=(\pi_{\Gamma_f,\Gamma_G})_*\omega_\mu\quad\text{in }
 M^1(\bP^1(\bC))^{\dag}=M^1(\Gamma_G)^{\dag},\label{eq:converse}
\end{align}
identifying $M^1(\Gamma_G)^{\dag}$ with $M^1(\bP^1(\bC))^{\dag}$ under 
the bijection
\begin{gather*}
 S(\Gamma_G)\setminus\{\cS_G\}
=T_{\cS_G}\sP^1(\bL)\cong\bP^1(k_{\bL})=\bP^1(\bC).
\end{gather*}

The following direct translation from degenerating complex dynamics 
into quantized Berkovich dynamics is 
based on the above explicit definition of $\omega_{\mu}$
and bypasses a correspondence between semistable models of
$\sP^1(\bL)$ and semistable vertex sets in $\sP^1(\bL)$ from
rigid analytic geometry (see, e.g, \cite{BPR13}),
which is used in \cite{DF14}.
See Section \ref{sec:complement} for a complement of this proposition.

\begin{proposition}[a direct translation, cf.\ {\cite[Proposition 5.1(1)]{DF14}}]\label{th:transfer}
For every ordered pair 
$\mu=(\mu_C,\mu_E)\in (M^1(\bP^1(\bC)))^2$
satisfying the admissibility
\eqref{eq:compatibility} $($for $A)$, we have
\begin{gather}
 \bigl(\widetilde{A\circ f}\bigr)^*\mu_E=d\cdot\mu_C\quad\text{in }M(\bP^1(\bC))
\quad\Rightarrow\quad
 f_G^*\omega_{\mu}=d\cdot(\pi_{\Gamma_f,\Gamma_G})_*\omega_{\mu}\quad\text{in }M(\Gamma_G).\label{eq:tosurface}
\end{gather} 
\end{proposition}

\begin{proof}
Pick an ordered pair $\mu=(\mu_C,\mu_E)\in (M^1(\bP^1(\bC)))^2$
satisfying the admissibility \eqref{eq:compatibility}
(for $A$), and write $\mu_C=\nu_C+\tilde{\nu}_C,\mu_E=\nu_E+\tilde{\nu}_E$ as in
Definition \ref{th:direct}.

{\bfseries (a-1).} 
When $\Gamma_f\neq\Gamma_G$, 
for every $\tilde{x}\in\bP^1(\bC)=\bP^1(k_{\bL})\cong T_{\cS_G}\sP^1(\bL)=S(\Gamma_G)\setminus\{\cS_G\}$ (and every representative $x\in\bP^1(\bL)$ of $\tilde{x}$), 
recalling Definitions \ref{th:multfactor} and \ref{th:direct},
we compute both
\begin{align*}
&\bigl(f_G^*\omega_{\mu}\bigr)(\{U_{\widevec{\cS_Gx}}\})\\
&\underset{\eqref{eq:pullbackdefining}}{=}
\int_{\sP^1/S(\Gamma_f)}m_{V,U_{\widevec{\cS_Gx}}}(f)\omega_{\mu}(V)\\
&=s_{\widevec{\cS_Gx}}(f)\cdot 1+m_{\widevec{\cS_Gx}}(f)\cdot
\omega_{\mu}\bigl(\bigl\{V\in S(\Gamma_f):V\subset U_{f_*(\widevec{\cS_Gx})}\bigr\}\bigr)\\
&=s_{\widevec{\cS_Gx}}(f)+m_{\widevec{\cS_Gx}}(f)\times\\
\times&
\begin{cases}
1-\omega_{\mu}\Bigl(\bigl\{U_{\vec{w}}\in S(\Gamma_f):\vec{w}\in(T_{f(\cS_G)}\sP^1(\bL))\setminus\{f_*(\widevec{\cS_Gx})\}\bigr\}\cup\{\{f(\cS_G)\}\}\Bigr)\\
\hspace*{25pt}
=1-\omega_{\mu}\bigl(\bigl\{U_{(A^{-1})_*(\widevec{\cS_Gy})}:y\in\bP^1(\bL)\text{ satisfying }\tilde{y}\in\bP^1(k_{\bL})\setminus\{a_A\}\bigr\}\bigr)\\
\hspace*{35pt}-\omega_{\mu}(\{\{f(\cS_G)\}\})\\
\hspace*{25pt}
=\mu_E(\{a_A\})\quad\text{if }f_*(\widevec{\cS_Gx})=\widevec{f(\cS_G)\cS_G},\\
\omega_{\mu}\bigl(\{U_{f_*(\widevec{\cS_Gx})}\}\bigr)
=\omega_{\mu}\bigl(\{U_{(A^{-1})_*(\widevec{\cS_Gy})}\}\bigr)\\
\hspace*{35pt}\text{for any such }y\in\bP^1(\bL)\text{ that }
 f_*(\widevec{\cS_Gx})=(A^{-1})_*(\widevec{\cS_Gy})\quad\text{otherwise}
\end{cases}\\
&\overset{\eqref{eq:intersection}}{=}
s_{\widevec{\cS_Gx}}(f)+m_{\widevec{\cS_Gx}}(f)\cdot\mu_E(\{\tilde{y}\})\\
&
\text{for any such }y\in\bP^1(\bL)\text{ that }
 (A\circ f)_*(\widevec{\cS_Gx})=\widevec{\cS_Gy}
\bigl(\Leftrightarrow f_*(\widevec{\cS_Gx})=(A^{-1})_*(\widevec{\cS_Gy})\bigr)\\
&\underset{\eqref{eq:surplusfaber},\eqref{directdegdef}\&\eqref{eq:tangentreduct}}{=}
\ord_{\tilde{x}}\bigl[H_{\widetilde{A\circ f}}=0\bigr]
+\bigl(\deg_{\tilde{x}}(\phi_{\widetilde{A\circ f}})\bigr)\cdot
\mu_E\bigl(\bigl\{\phi_{\widetilde{A\circ f}}(\tilde{x})\bigr\}\bigr)\\
&\underset{\eqref{eq:degenpullback}}{=}\bigl((\widetilde{A\circ f})^*\mu_E\bigr)(\{\tilde{x}\})
\end{align*}
and
\begin{align*}
&\bigl((\pi_{\Gamma_f,\Gamma_G})_*\omega_{\mu}\bigr)(\{U_{\widevec{\cS_Gx}}\})\overset{\eqref{eq:factorrel}}{=}
\omega_{\mu}\bigl(\bigl\{V\in S(\Gamma_f):V\subset U_{\widevec{\cS_Gx}}\bigr\}\bigr)\\
&=
\begin{cases}
1-\omega_{\mu}\Bigl(\bigl\{U_{\vec{v}}\in S(\Gamma_f):\vec{v}\in(T_{\cS_G}\sP^1(\bL))\setminus\{\widevec{\cS_Gx}\}\bigr\}\cup\{\{\cS_G\}\}\Bigr)\\
\hspace*{20pt}=1-\mu_C\bigl(\bP^1(\bC)\setminus\{h_A\}\bigr)
=\mu_C(\{h_A\})
\quad\text{if }\widevec{\cS_Gx}=\widevec{\cS_Gf(\cS_G)},\\
 \omega_{\mu}(\{U_{\widevec{\cS_Gx}}\})
\quad\text{otherwise}  
 \end{cases}\\
&\underset{\eqref{eq:intersection}}{=}\mu_C(\{\tilde{x}\}).
\end{align*}
Hence if $(\widetilde{A\circ f})^*\mu_E=d\cdot\mu_C$ on $\bP^1(\bC)$, then 
we have the equality
$(f_G^*\omega_{\mu})(\{U_{\widevec{\cS_Gx}}\})
=(d\cdot(\pi_{\Gamma_f,\Gamma_G})_*\omega_{\mu})(\{U_{\widevec{\cS_Gx}}\})$.

{\bfseries (a-2).}
Moreover, we also compute both
\begin{align*}
\bigl(f_G^*\omega_{\mu}\bigr)(\{\{\cS_G\}\}) 
\underset{\eqref{eq:pullbackdefining}}{=}&
\int_{\sP^1/S(\Gamma_f)}m_{V,\{\cS_G\}}(f)\omega_{\mu}(V)
=\deg_{\cS_G}(f)\cdot\omega_{\mu}(\{\{f(\cS_G)\}\})\\
\underset{\eqref{eq:totallocaldegree}}{=}&
\deg\bigl(\phi_{\widetilde{A\circ f}}\bigr)\cdot\nu_E\bigl(\bP^1(\bC)\bigr)
=\bigl((\widetilde{A\circ f})^*\mu_E\bigr)\bigl(\bP^1(\bC)\setminus F_1\bigr)
\end{align*}
and
$((\pi_{\Gamma_f,\Gamma_G})_*\omega_{\mu})(\{\{\cS_G\}\})
\underset{\eqref{eq:factorrel}}{=}\omega_{\mu}(\{\{\cS_G\}\})=\nu_C(\bP^1(\bC))
=\mu_C(\bP^1(\bC)\setminus F_2)$,
where $F_1,F_2$ are any sufficiently large countable subsets in $\bP^1(\bC)$.

Hence if $(\widetilde{A\circ f})^*\mu_E=d\cdot\mu_C$ on $\bP^1(\bC)$,
then we also have $(f_G^*\omega_{\mu})(\{\{\cS_G\}\})
=(d\cdot(\pi_{\Gamma_f,\Gamma_G})_*\omega_{\mu})(\{\{\cS_G\}\})$.
Now the proof is complete in this case.

{\bfseries (b-1).} 
When $\Gamma_f=\Gamma_G$, 
for every $\tilde{x}\in\bP^1(\bC)=\bP^1(k_{\bL})\cong T_{\cS_G}\sP^1(\bL)=S(\Gamma_G)\setminus\{\cS_G\}$, similarly to {\bfseries (a-1)}, we compute both
\begin{align*}
&\bigl(f_G^*\omega_{\mu}\bigr)(\{U_{\widevec{\cS_Gx}}\})\\
&\underset{\eqref{eq:pullbackdefining}}{=}
\int_{\sP^1/S(\Gamma_G)}m_{V,U_{\widevec{\cS_Gx}}}(f)\omega_{\mu}(V)
=s_{\widevec{\cS_Gx}}(f)\cdot 1+
m_{\widevec{\cS_Gx}}(f)\cdot \omega_{\mu}(\{U_{f_*(\widevec{\cS_Gx})}\})\\
&=s_{\widevec{\cS_Gx}}(f)+m_{\widevec{\cS_Gx}}(f)\cdot
\mu_E(\{\tilde y\})\\
&\hspace*{50pt}\text{for any such }y\in\bP^1(\bL)
\text{ that }f_*(\widevec{\cS_Gx})=(A^{-1})_*(\widevec{\cS_Gy})\\
&\underset{\eqref{eq:surplusfaber},\eqref{directdegdef}\&\eqref{eq:tangentreduct}}{=}
\ord_{\tilde{x}}\bigl[H_{\widetilde{A\circ f}}=0\bigr]
+\bigl(\deg_{\tilde{x}}(\phi_{\widetilde{A\circ f}})\bigr)\cdot
\mu_E\bigl(\bigl\{\phi_{\widetilde{A\circ f}}(\tilde{x})\bigr\}\bigr)\\
&\underset{\eqref{eq:degenpullback}}{=}\bigl((\widetilde{A\circ f})^*\mu_E\bigr)(\{\tilde{x}\})
\end{align*}
and
\begin{gather*}
 \bigl((\pi_{\Gamma_f,\Gamma_G})_*\omega_{\mu}\bigr)(\{U_{\widevec{\cS_Gx}}\})
\underset{\eqref{eq:factorrel}}{=}\omega_{\mu}(\{U_{\widevec{\cS_Gx}}\})
=\mu_C(\{\tilde{x}\}).
\end{gather*}
Hence if $(\widetilde{A\circ f})^*\mu_E=d\cdot\mu_C$ on $\bP^1(\bC)$, then we have
the equality
$(f_G^*\omega_{\mu})(\{U_{\widevec{\cS_Gx}}\})=(d\cdot(\pi_{\Gamma_f,\Gamma_G})_*\omega_{\mu})(\{U_{\widevec{\cS_Gx}}\})$.

{\bfseries (b-2).} 
Similarly to {\bfseries (a-2)}, we also compute both
\begin{align*}
\bigl(f_G^*\omega_{\mu}\bigr)(\{\{\cS_G\}\}) 
\underset{\eqref{eq:pullbackdefining}}{=}&
\int_{\sP^1/S(\Gamma_G)}m_{V,\{\cS_G\}}(f)\omega_{\mu}(V)
=\deg_{\cS_G}(f)\cdot\omega_{\mu}(\{\{\cS_G\}\})\\
\underset{\eqref{eq:totallocaldegree}}{=}&
\deg\bigl(\phi_{\widetilde{A\circ f}}\bigr)\cdot\nu_E\bigl(\bP^1(\bC)\bigr)
=\bigl((\widetilde{A\circ f})^*\mu_E\bigr)\bigl(\bP^1(\bC)\setminus F_1\bigr)
\end{align*}
and 
$\bigl((\pi_{\Gamma_f,\Gamma_G})_*\omega_{\mu}\bigr)(\{\{\cS_G\}\})
\underset{\eqref{eq:factorrel}}{=}\omega_{\mu}(\{\{\cS_G\}\})
=\nu_C(\bP^1(\bC))
=\mu_C(\bP^1(\bC)\setminus F_2)$,
where $F_1,F_2$ are any sufficiently large countable subsets in $\bP^1(\bC)$.

Hence if $(\widetilde{A\circ f})^*\mu_E=d\cdot\mu_C$ on $\bP^1(\bC)$,
then we also have $(f_G^*\omega_{\mu})(\{\{\cS_G\}\})
=(d\cdot(\pi_{\Gamma_f,\Gamma_G})_*\omega_{\mu})(\{\{\cS_G\}\})$.
Now the proof is also complete in this case.
\end{proof}

The following complements Theorem \ref{th:balancedgeneral}.

\begin{proposition}\label{th:admissible}
If $\mu_C=\lim_{j\to\infty}\mu_{f_{t_j}},\mu_E=\lim_{j\to\infty}(A_{t_j})_*\mu_{f_{t_j}}$ are weak limit points on $\bP^1(\bC)$ as $t\to 0$ of 
$(\mu_{f_t})_{t\in\bD^*},((A_t)_*\mu_{f_t})_{t\in\bD^*}$, 
respectively, for some sequence $(t=t_j)$ in $\bD^*$ tending to 
$0$ as $j\to\infty$, then $\mu:=(\mu_C,\mu_E)\in(M^1(\bP^1(\bC))^\dag)^2$
also satisfies the admissibility \eqref{eq:compatibility} $($for $A)$.
\end{proposition}

\begin{proof}
When $\Gamma_f=\Gamma_G$ or equivalently $\tilde{A}=\phi_{\tilde{A}}$,
by the uniform convergence \eqref{eq:locunifoutside} and $\supp[H_{\tilde{A}}=0]=\emptyset$, 
we have $\tilde{A}_*\mu_C=\mu_E$
on $\bP^1(\bC)$, that is, the admissibility
$\tilde{A}^*\mu_E=\mu_C$ on $\bP^1(\bC)$ in this case holds.

When $\Gamma_f\neq\Gamma_G$, for $0<\epsilon\ll 1$,
by the outer regularity of $\mu_E$,
there is a continuous test function $\psi$ on $\bP^1(\bC)$
such that $\psi\ge 0$ on $\bP^1(\bC)$, that $\psi\equiv 1$
on an open neighborhood 
of $a_A$, and that
$\mu_E(\{a_A\})+\epsilon/2>\int_{\bP^1(\bC)}\psi\mu_E$.
Then for any continuous test function $\eta$ on $\bP^1(\bC)$
supported by $\bP^1(\bC)\setminus\{h_A\}$ and
satisfying $0\le\eta\le 1$ on $\bP^1(\bC)$, 
we have 
$\mu_E(\{a_A\})+\epsilon>\int_{\bP^1(\bC)}\psi((A_{t_j})_*\mu_{f_{t_j}})
=\int_{\bP^1(\bC)}(\psi\circ A_{t_j})\mu_{f_{t_j}}
\ge\int_{\supp\eta}(\psi\circ A_{t_j})\cdot\eta\mu_{f_{t_j}}$
for $j\gg 1$.
Then by the uniform convergence \eqref{eq:locunifoutside} and
the first item in \eqref{eq:redconst}, 
we even have $\mu_E(\{a_A\})+\epsilon
>\int_{\supp\eta}1\cdot\eta\mu_{f_{t_j}}=\int_{\bP^1(\bC)}\eta\mu_{f_{t_j}}$
for $j\gg 1$, so that $\mu_E(\{a_A\})+\epsilon\ge\int_{\bP^1(\bC)}\eta\mu_C$
making $j\to\infty$.
Hence by the inner regularity of $\mu_C$,
we have $\mu_E(\{a_A\})+2\epsilon\ge\mu_C(\bP^1(\bC)\setminus\{h_A\})$, and in turn
$\mu_E(\{a_A\})\ge\mu_C(\bP^1(\bC)\setminus\{h_A\})$,
that is, the admissibility $\mu_C(\{h_A\})+\mu_E(\{a_A\})\ge(\mu_C(\bP^1(\bC))=)1$
in this case also holds.
\end{proof}

\section{Proof of Theorem \ref{th:computation}}
\label{sec:proofmain}

Let $K$ be an algebraically closed field that is complete
with respect to a non-trivial and non-archimedean absolute value $|\cdot|$,
and let $f\in K(z)$ be a rational function on $\bP^1$ of $\deg f=:d>1$.
Recall that 
\begin{gather*}
\Gamma_G:=\{\cS_G\}\quad\text{and}\quad
\Gamma_n:=\Gamma_{f^n}:=\bigl\{\cS_G,f^n(\cS_G)\bigr\}\quad\text{in }\sH^1_{\mathrm{II}}
\end{gather*}
for each $n\in\bN$,
and the definitions of $\Delta_f,\Delta_f^\dag$ in Subsection \ref{sec:mainth}. 

\begin{lemma}[cf.\ {\cite[Lemma 4.4]{DF14}}]\label{th:balancednonarchi}
For every $\nu\in M^1(\sP^1)$, if $\nu$ has the $f$-balanced
property $f^*\nu=d\cdot\nu$ on $\sP^1$ and satisfies $\nu(\{f(\cS_G)\})=0$, then
for every $n\in\bN$, $(\pi_{\Gamma_n})_*\nu\in M^1(\Gamma_n)$
has the quantized $f^n$-balanced property $($see \eqref{eq:quantizedbalanced}$)$, 
and if in addition $f^{-1}(\cS_G)\neq\{\cS_G\}$, then 
$(\pi_{\Gamma_G})_*\nu\in\Delta_f^\dag$.
\end{lemma}

\begin{proof}
Under the assumption on $\nu$, for every $U\in S(\Gamma_G)\setminus\{\{\cS_G\}\}$,
we compute
\begin{multline*}
\bigl(f_G^*((\pi_{\Gamma_f})_*\nu)\bigr)(\{U\})
=\bigl\langle(\pi_{\Gamma_f})^*\bigl(f_{G,*}1_{\{U\}}\bigr),\nu\bigr\rangle\\
\underset{\eqref{eq:pushfactor}}{=}
\int_{\sP^1\setminus\{f(\cS_G)\}}\bigl((f^*\delta_{\,\cdot\,})(U)\bigr)\nu
=\bigl\langle(f^*\delta_{\,\cdot\,})(U),\nu\bigr\rangle
=\bigl\langle 1_U,f^*\nu\bigr\rangle\\
=\bigl\langle (\pi_{\Gamma_G})^*1_{\{U\}},d\cdot\nu\bigr\rangle
=d\cdot\bigl((\pi_{\Gamma_G})_*\nu\bigr)(\{U\})
=d\cdot\bigl((\pi_{\Gamma_f,\Gamma_G})_*((\pi_{\Gamma_f})_*\nu)\bigr)(\{U\}),
\end{multline*}
so that also recalling \eqref{eq:countablysupp}, 
$(\pi_{\Gamma_f})_*\nu\in M^1(\Gamma_f)$ has the quantized $f$-balanced property
\eqref{eq:quantizedbalanced}.
On the other hand, for any $n\in\bN$, we have $(f^n)^*\nu=d^n\cdot\nu$ on $\sP^1$, 
and in turn 
\begin{multline*}
 0=d^{n-1}\cdot(\deg_{\cS_G}f)\cdot\nu(\{f(\cS_G)\})=
d^{n-1}\cdot(f^*\nu)(\{\cS_G\})\\
=((f^n)^*\nu)(\{\cS_G\})
=\deg_{\cS_G}(f^n)\cdot\nu(\{f^n(\cS_G)\})\ge\nu(\{f^n(\cS_G)\})(\ge 0),
\end{multline*}
so $\nu(\{f^n(\cS_G)\})=0$. Hence the former assertion holds, and
so does the latter 
also by \eqref{eq:nopotgood}, \eqref{eq:countablysupp}, \eqref{eq:canonicalpush}(, and $(\pi_{\Gamma_G})_*=(\pi_{\Gamma_n,\Gamma_G})_*(\pi_{\Gamma_n})_*$).
\end{proof}

\begin{proof}[Proof of Theorem {\rm \ref{th:computation}}]
Suppose that $f^{-1}(\cS_G)\neq\{\cS_G\}$, which is equivalent to that
\begin{gather}
 \nu_f(\{f(\cS_G)\})=\nu_f(\{\cS_G\})=
\bigl((\pi_{\Gamma_G})_*\nu_f\bigr)(\{\{\cS_G\}\})=0\label{eq:vanish}
\end{gather}
(by \eqref{eq:nopotgood}). Then by Lemma \ref{th:balancednonarchi},
we have $(\pi_{\Gamma_G})_*\nu_f\in\Delta_f^\dag$.
Suppose also that $\Char K=0$ (so $\#E(f)\le 2$)
and, in turn, that for any $a\in E(f)$, $f(a)=a$ or equivalently $f^{-1}(a)=\{a\}$.
Then for every $a\in E(f)$, by Lemma \ref{th:balancednonarchi},
we also have $(\pi_{\Gamma_G})_*\delta_a\in\Delta_f^\dag$.
Moreover,
for every $a\in E(f)$, 
every $n\in\bN$, and every 
$\vec{v}\in(T_{\cS_G}\sP^1)\setminus\{\widevec{\cS_Ga}\}$,
by Facts \ref{th:surplus} and \ref{th:directionalpullback},
we have
\begin{gather}
s_{\vec{v}}(f^n)=0\, \bigl(\Leftrightarrow f^n(U_{\vec{v}})=U_{(f^n)_*\vec{v}}
\bigr)\quad\text{and}\label{eq:nonexceptionaldirection}\\
(f^n)_*(\vec{v})\neq\widevec{f^n(\cS_G)a},
\label{eq:nonexceptionaldirectionsurplus}
\end{gather}
and for every $a\in E(f)$ and every $n\in\bN$, 
we also have
\begin{gather}
s_{\widevec{\cS_Ga}}(f^n)=d^n-\deg_{\cS_G}(f^n)\quad 
(\text{also using }\eqref{eq:totalsurplus})\quad\text{and}\label{eq:excepsurplus}\\
 (f^n)_*\bigl(\widevec{\cS_Ga}\bigr)=\widevec{f^n(\cS_G)a}
\quad(\text{also using Fact \ref{th:directionalpullback}}).\label{eq:exceptionaldirection}
\end{gather}

{\bfseries (a).} 
Let us see the former half in Theorem \ref{th:computation}.
If for any $\vec{v}\in T_{\cS_G}\sP^1$, we have
\begin{gather}
\limsup_{n\to\infty}\frac{s_{\vec{v}}(f^n)}{d^n}\ge\nu_f(U_{\vec{v}})\Bigl(=\bigl((\pi_{\Gamma_G})_*\nu_f\bigr)(\{U_{\vec{v}}\})\Bigr),\label{eq:surplusequidist}
\end{gather}
then for every $\omega\in\Delta_f$ and $n\gg 1$, 
fixing $\omega_n\in M^1(\Gamma_n)$ such that
$\omega_n(S(\Gamma_n)\setminus F)=0$ for some countable subset $F$ in
$S(\Gamma_n)$ and that
$d^{-n}(f^n)_G^*\omega_n=\omega(=(\pi_{\Gamma_n,\Gamma_G})_*\omega_n)$ 
in $M^1(\Gamma_G)$, 
also recalling Definition \ref{th:multfactor},
for every $\vec{v}\in T_{\cS_G}\sP^1$, we have 
\begin{multline*}
\omega(\{U_{\vec{v}}\})
=\limsup_{n\to\infty}\frac{\bigl((f^n)_G^*\omega_n\bigr)(\{U_{\vec{v}}\})}{d^n}
\underset{\eqref{eq:pullbackdefining}}{\ge}
\limsup_{n\to\infty}\Bigl(\frac{s_{\vec{v}}(f^n)}{d^n}
\cdot\omega_n\bigl(\sP^1/S(\Gamma_n)\bigr)\Bigr)\\
=\limsup_{n\to\infty}\frac{s_{\vec{v}}(f^n)}{d^n}
\ge
\bigl((\pi_{\Gamma_G})_*\nu_f\bigr)(\{U_{\vec{v}}\}),
\end{multline*} 
which with \eqref{eq:vanish} 
and \eqref{eq:countablysupp}
yields $\omega=(\pi_{\Gamma_G})_*\nu_f$ in $M^1(\Gamma_G)$.
Hence we have
$\Delta_f=\Delta_f^\dag=\{(\pi_{\Gamma_G})_*\nu_f\}$,
i.e., the case (i) in Theorem \ref{th:computation} holds, 
under this ``surplus equidistribution'' assumption 
\eqref{eq:surplusequidist}
(see \cite[p.\ 27]{DF14}).

{\bfseries (b.1).} 
Alternatively, suppose that 
there is $\vec{u}\in T_{\cS_G}\sP^1$ 
not satisfying \eqref{eq:surplusequidist}. 
Then 
fixing any $\cS\in\bP^1\setminus E(f)(\subset\sP^1\setminus(E(f)\cup\{f^n(\cS_G):n\in\bN\}))$, 
we have
\begin{multline*}
\notag\nu_f(U_{\vec{u}})
\le\limsup_{n\to\infty}\frac{\bigl((f^n)^*\delta_{\cS}\bigr)(U_{\vec{u}})}{d^n}\\
\underset{\eqref{eq:argumentnonarchi}}{\le}
\limsup_{n\to\infty}\frac{m_{\vec{u}}(f^n)}{d^n}
+\limsup_{n\to\infty}\frac{s_{\vec{u}}(f^n)}{d^n}
<\limsup_{n\to\infty}\frac{m_{\vec{u}}(f^n)}{d^n}+\nu_f(U_{\vec{u}}),
\end{multline*}
the first inequality in which is by 
the inner regularity of $\nu_f$ and \eqref{eq:equidist}(, and
the equality holds if $\cS_G\in\sP^1\setminus\sJ(f)$).
Hence $0<\limsup_{n\to\infty}(m_{\vec{u}}(f^n)/d^n)
=\prod_{j=0}^\infty(m_{(f^j)_*(\vec{u})}(f)/d)$,
so that $m_{(f^n)_*(\vec{u})}(f)\equiv d(>1)$ for $n\gg 1$,
and in turn, also recalling \eqref{eq:totallocaldegree}
(and the maximally ramification locus $\sR_{\max}(f)$ of $f$ in Subsection \ref{sec:balanced}), that
\begin{gather}
\deg_{f^n(\cS_G)}(f)\equiv d,\text{ i.e., }f^n(\cS_G)\in\sR_{\max}(f),\quad\text{for }n\gg 1;\label{eq:eventual} 
\end{gather}
then $f^{n+1}(\cS_G)\neq f^n(\cS_G)$ for $n\gg 1$
under the assumption that $f^{-1}(\cS_G)\neq\{\cS_G\}$.

Recall also that $\sR_{\max}(f)$ of $f$ is connected in $\sP^1$. 
Hence for $n\gg 1$, 
%we have 
$f^{-1}([f^n(\cS_G),f^{n+1}(\cS_G)])=[f^{n-1}(\cS_G),f^n(\cS_G)]
\subset\sR_{\max}(f)$, 
and then
$f$ restricts to a homeomorphism from 
$[f^{n-1}(\cS_G),f^n(\cS_G)]$ onto $[f^n(\cS_G),f^{n+1}(\cS_G)]$
and, recalling \eqref{eq:totallocaldegree}, we also have
$\cS\mapsto m_{\widevec{\cS f^n(\cS_G)}}(f)=\deg_{\cS}(f)\equiv d(>1)$ on $[f^{n-1}(\cS_G),f^n(\cS_G)]$. Then for any $m\ge n\gg 1$, 
$\rho(f^m(\cS_G),f^{m+1}(\cS_G))=d^{m-n}\cdot\rho(f^n(\cS_G),f^{n+1}(\cS_G))$
by \eqref{eq:Lipschitz}.
Consequently, also by the upper semicontinuity of $\deg_{\,\cdot}(f)$ on $\sP^1$,
there is $a\in\bP^1$ such that
\begin{gather*}
 \bigl\{f^n(a):n\in\bN\cup\{0\}\bigr\}
\subset\bigl(\bP^1\cap \sR_{\max}(f)\bigr)\cap\bigcap_{N\in\bN}\overline{\{f^n(\cS_G):n\ge N\}}, 
\end{gather*}
which with $\#(\bP^1\cap \sR_{\max}(f))\le 2$ 
(mentioned in \eqref{eq:maximaldegree}) still implies
\begin{gather*}
 a\in E(f).
\end{gather*}
Under the assumption that $f(a)=a$ (or equivalently
$f^{-1}(a)=\{a\}$ so $f'(a)=0$),
we conclude that $\lim_{n\to\infty}f^n(\cS_G)=a$
(and $\cS_G\in\sP^1\setminus\sJ(f)$) 
and, moreover, that $f^n(\cS_G)\in(\cS_G,a]$ for $n\gg 1$,
using \cite[Theorem F]{Faber13topologyII} and \eqref{eq:Lipschitz}
(see \cite[p.\ 25]{DF14})(, or now assuming that $f$ is tamely maximally 
ramified near this $a\in E(f)\subset\sR_{\max}(f)\cap\bP^1$, for simplicity).

\begin{remark}
 Conversely, if there is such $a\in E(f)$ that $\lim_{n\to\infty}f^n(\cS_G)=a$
 and that $f^n(\cS_G)\in(\cS_G,a]$ for $n\gg 1$, then 
 \eqref{eq:eventual} is the case (since there is 
 $\cS\in(a,\cS_G]$ so close to $a$ that $(a,\cS]\subset\sR_{\max}(f)$), 
 and \eqref{eq:eventual}
 together with \eqref{eq:nonexceptionaldirection} and \eqref{eq:excepsurplus}
 implies that
 the inequality \eqref{eq:surplusequidist} for this $a$ does not hold for some 
 $\vec{v}\in T_{\cS_G}\sP^1$. 
\end{remark}

{\bfseries (b.2).} 
Once such $a\in E(f)$ is at our disposal,
for $n\gg 1$, 
noting that $f^{-1}(a)=\{a\}$, that $\lim_{n\to\infty}f^n(\cS_G)=a$,
and that $f^n(\cS_G)\in(\cS_G,a]$ for $n\gg 1$, 
we have 
\begin{gather}
 f\bigl(U_{\widevec{f^n(\cS_G)a}}\bigr)
=U_{\widevec{f^{n+1}(\cS_G)a}}\label{eq:superattbasin}
\end{gather}
(also by Fact \ref{th:surplus} applied to $\widevec{f^n(\cS_G)a}\in T_{f^n(\cS_G)}\sP^1$)
and not only
\begin{gather}
 \nu_f\bigl(U_{\widevec{f^n(\cS_G)\cS_G}}\bigr)=1
\label{eq:total}
\end{gather}
for $n\gg 1$, but also $\cS_G\in\sP^1\setminus\sJ(f)$
(also since $a\in\sP^1\setminus\sJ(f)$ and $f(\sJ(f))=\sJ(f)$),
and then fixing such $n_0\gg 1$ that
$\deg_{\cS_G}(f^n)/d^n$ is constant for $n\ge n_0$ 
(by \eqref{eq:eventual}) and 
fixing any $\cS\in\sP^1\setminus E(f)$,
for every $n\ge n_0$, we also have
\begin{align}
\notag 0<&\frac{\deg_{\cS_G}(f^n)}{d^n}
\biggl(\underset{\eqref{eq:excepsurplus}}{=}\frac{d^n-s_{\widevec{\cS_G a}}(f^n)}{d^n}=\\
\notag&\underset{\eqref{eq:exceptionaldirection}\&\eqref{eq:argumentnonarchi},\text{ when }n\gg 1}{=}1-\frac{(f^n)^*\delta_{\cS}}{d^n}(U_{\widevec{\cS_Ga}})
\equiv 1-\limsup_{n\to\infty}\frac{(f^n)^*\delta_{\cS}}{d^n}(U_{\widevec{\cS_Ga}})=\biggr)\\
&\underset{\eqref{eq:equidist}\&\cS_G\in\sP^1\setminus\sJ(f)}{\equiv}
1-\nu_f(U_{\widevec{\cS_Ga}});
\label{eq:stationary}
\end{align}
in particular $\nu_f(U_{\widevec{\cS_Ga}})<1$, 
and in turn $(\pi_{\Gamma_G})_*\nu_f\neq(\pi_{\Gamma_G})_*\delta_a$.

Now the case (ii) in Theorem \ref{th:computation} holds
under this ``surplus {\itshape in}equidistribution'' assumption, 
and the proof of the former half in Theorem \ref{th:computation} is complete. 

\begin{remark}\label{th:notalways}
In \cite[\S4.6]{DF14}, the condition 
$\sJ(f)\subset\sP^1\setminus(U_{\widevec{\cS_Ga}}\cup\{\cS_G\})$ 
was assumed with loss of some generality;
under this condition, the vanishing assumption on each $\omega_n$
in the definition \eqref{eq:deltaquantized} of $\Delta_f$ does not matter
(and did not appear in \cite[\S4.6]{DF14}). 
By \eqref{eq:stationary} (and $\deg_{\,\cdot\,}(f)\in\{1,\ldots,d\}$), the statement
$\nu_f(U_{\widevec{\cS_Ga}})=0
\bigl(\Leftarrow\sJ(f)\subset\sP^1\setminus(U_{\widevec{\cS_Ga}}\cup\{\cS_G\})\bigr)$
is equivalent to that
\begin{gather}
 \deg_{f^n(\cS_G)}(f)\equiv d\quad\text{for any }n\in\bN\cup\{0\},\tag{\ref{eq:eventual}$'$}\label{eq:identically}
\end{gather}
and is indeed not always the case (as seen in Section \ref{sec:example} below).
\end{remark}

{\bfseries (c.1).} Let us show the latter half,
i.e., the equality \eqref{eq:computationdelta}, in Theorem \ref{th:computation}.
For $n\gg 1$, 
by \eqref{eq:total}, \eqref{eq:nonexceptionaldirection},
the $f^n$-balanced property of $\nu_f$ on $\sP^1$, 
and Fact \ref{th:surplus},
for every 
$\vec{v}\in(T_{\cS_G}\sP^1)\setminus\{\widevec{\cS_Ga}\}$,
we have the equivalence
\begin{gather}
 \nu_f(U_{\vec{v}})>0\Leftrightarrow
(f^n)_*(\vec{v})
=\widevec{f^n(\cS_G)\cS_G}
\Leftrightarrow f^n(U_{\vec{v}})=U_{\widevec{f^n(\cS_G)\cS_G}},\label{eq:nonexceptionalpositive}
\end{gather}
(one of) which is the case for at least one $\vec{v}\in(T_{\cS_G}\sP^1)\setminus\{\widevec{\cS_Ga}\}$ since $\nu_f(U_{\widevec{\cS_Ga}})<1$.
Hence for $n\gg 1$, using the $f^n$-balanced property of $\nu_f$ on $\sP^1$ again,
for every 
$\vec{v}\in(T_{\cS_G}\sP^1)\setminus\{\widevec{\cS_Ga}\}$
satisfying $\nu_f(U_{\vec{v}})>0$, 
we have
\begin{multline}
(0<)\nu_f(U_{\vec{v}})
=\frac{(f^n)^*\nu_f}{d^n}(U_{\vec{v}})
=\frac{1}{d^n}\int_{f^n(U_{\vec{v}})}\bigl((f^n)^*\delta_{\cS}\bigr)(U_{\vec{v}})\nu_f(\cS)\\
\underset{\eqref{eq:nonexceptionalpositive}\&\eqref{eq:argumentnonarchi}}{=}
\frac{m_{\vec{v}}(f^n)+s_{\vec{v}}(f^n)}{d^n}
\cdot\nu_f(U_{\widevec{f^n(\cS_G)\cS_G}})
\underset{\eqref{eq:nonexceptionaldirection}\&\eqref{eq:total}}{=}
\frac{m_{\vec{v}}(f^n)}{d^n}
\label{eq:nonexceptionaldirectionmeasure}
\end{multline}
(and $m_{(f^n)_*\vec{v}}(f)\equiv d$). On the other hand, for $n\gg 1$,
by \eqref{eq:nonexceptionaldirectionsurplus}, 
\eqref{eq:exceptionaldirection}, 
\eqref{eq:nonexceptionalpositive}, 
and Fact \ref{th:directionalpullback},
we have 
\begin{multline}
\bigl\{(f^n)_*(\vec{v}):
\vec{v}\in(T_{\cS_G}\sP^1)\setminus\{\widevec{\cS_Ga}\}
\text{ satisfying }\nu_f(U_{\vec{v}})=0\bigr\}\\
=\bigl(T_{f^n(\cS_G)}\sP^1\bigr)\setminus\bigl\{\widevec{f^n(\cS_G)a},\widevec{f^n(\cS_G)\cS_G}\bigr\}.
\label{eq:annulus}
\end{multline}
Now we assume that $f$ is tamely maximally ramified near this $a\in E(f)\subset\sR(f)\cap\bP^1$. Then there is $\cS\in(\cS_G,a]\setminus\{a\}$ such that 
$\sR_{\max}(f)\cap U_{\widevec{\cS a}}=(\cS,a]$, and in turn 
for every $\cS'\in(\cS,a]\setminus\{a\}$ and
every $\vec{w}=\widevec{\cS'\cS''}\in(T_{\cS'}\sP^1)\setminus\{\widevec{\cS' a},\widevec{\cS'\cS}\}$, 
diminishing $[\cS',\cS'']$ if necessary, 
we have 
$m_{\vec{w}}(f)
=m_{\widevec{\cS''\cS'}}(f)
\le\deg_{\cS''}(f)<d$
by \eqref{eq:Lipschitz} and \eqref{eq:totallocaldegree}.
Hence by \eqref{eq:annulus}, for $n\gg 1$, since 
$\lim_{n\to\infty}f^n(\cS_G)=a$ and $f^n(\cS_G)\in(\cS_G,a]$, 
for every $\vec{v}\in(T_{\cS_G}\sP^1)\setminus\{\widevec{\cS_Ga}\}$ 
satisfying $\nu_f(U_{\vec{v}})=0$, 
we have
\begin{gather}
 m_{(f^n)_*(\vec{v})}(f)\le d-1.\label{eq:simpledirect}
\end{gather}

{\bfseries (c.2).} 
Pick $\omega\in\Delta_f$ and, for $n\gg 1$,
fix $\omega_n\in M^1(\Gamma_n)$ satisfying
$\omega_n(S(\Gamma_n)\setminus F)=0$ for some countable subset $F$ in
$S(\Gamma_n)$ and 
$d^{-n}(f^n)_G^*\omega_n=\omega=(\pi_{\Gamma_n,\Gamma_G})_*\omega_n$ in $M^1(\Gamma_G)$. Then by 
the latter equality
$\omega=(\pi_{\Gamma_n,\Gamma_G})_*\omega_n$,
$\omega$ also satisfies 
$\omega(S(\Gamma_G)\setminus F)=0$ for some countable subset $F$ in
$S(\Gamma_G)$. 

Let us compute $\omega(\{U\})$ for each $U\in S(\Gamma_G)$.
For $n\gg 1$, using 
the equality
$d^{-n}(f^n)_G^*\omega_n=\omega$
(and recalling Definition \ref{th:multfactor}), we have both
\begin{multline}
\omega(\{U_{\vec{v}}\})
\underset{\eqref{eq:nonexceptionaldirection}\&\eqref{eq:pullbackdefining}}{=}
\frac{m_{\vec{v}}(f^n)}{d^n}\cdot
\omega_n\bigl(\bigl\{V\in S(\Gamma_n):V\subset U_{(f^n)_*(\vec{v})}\bigr\}\bigr)\\
\text{for any }
\vec{v}\in(T_{\cS_G}\sP^1)\setminus\{\widevec{\cS_Ga}\}
\label{eq:pullbacknonexcep}
\end{multline}
and
\begin{gather}
 \omega_n(\{\{f^n(\cS_G)\}\})
\underset{\eqref{eq:pullbackdefining}}{=}
\frac{d^n\cdot\omega(\{\{\cS_G\}\})}{\deg_{\cS_G}(f^n)}
\underset{\eqref{eq:stationary}}{=}
\frac{\omega(\{\{\cS_G\}\})}{1-\nu_f(U_{\widevec{\cS_Ga}})}.\label{eq:gaussimagemeas}
\end{gather}
Then for $n\gg 1$,
by \eqref{eq:pullbacknonexcep}, 
\eqref{eq:nonexceptionalpositive}, and \eqref{eq:nonexceptionaldirectionmeasure}, 
we have
$\omega_n(\{V\in S(\Gamma_n):V\subset U_{\widevec{f^n(\cS_G)\cS_G}}\})
=\omega(\{U_{\vec{v}}\})/\nu_f(U_{\vec{v}})$
for every $\vec{v}\in(T_{\cS_G}\sP^1)\setminus\{\widevec{\cS_Ga}\}$
satisfying $\nu_f(U_{\vec{v}})>0$.
Hence there exists a constant $s_\omega\in[0,1]$ such that for $n\gg 1$,
\begin{gather}
\omega_n\bigl(\bigl\{V\in S(\Gamma_n):V\subset U_{\widevec{f^n(\cS_G)\cS_G}}\bigr\}\bigr)\equiv s_\omega\label{eq:eventualfull} 
\end{gather}
and that for every $\vec{v}\in(T_{\cS_G}\sP^1)\setminus\{\widevec{\cS_Ga}\}$ satisfying $\nu_f(U_{\vec{v}})>0$,
\begin{gather}
\omega(\{U_{\vec{v}}\})=s_\omega\nu_f(U_{\vec{v}}).
\label{eq:nonexceptionaldirectionpositivemeasure}
\end{gather}
Moreover, for every 
$\vec{v}\in(T_{\cS_G}\sP^1)\setminus\{\widevec{\cS_Ga}\}$
satisfying $\nu_f(U_{\vec{v}})=0$, we have
\begin{gather*}
0\le\omega(\{U_{\vec{v}}\})
\underset{\eqref{eq:pullbacknonexcep}}{\le}
\frac{m_{\vec{v}}(f^n)}{d^n}\cdot 1
=\prod_{j=0}^{n-1}\frac{m_{(f^j)_*(\vec{v})}(f)}{d}
\underset{\eqref{eq:simpledirect}}{\to} 0\quad\text{as }n\to\infty, 
\end{gather*}
so we still have
\begin{gather}
 \omega(\{U_{\vec{v}}\})=0=s_\omega\nu_f(U_{\vec{v}}).\label{eq:nonexceptionaldirectionzeromeasure} 
\end{gather}
Now for $n\gg 1$, 
we also have
\begin{align}
\notag \omega\bigl(\{U_{\widevec{\cS_Ga}}\}\bigr)
&=1-\omega\bigl(\bigl\{U_{\vec{v}}\in S(\Gamma_G):\vec{v}\in(T_{\cS_G}\sP^1)\setminus\{\widevec{\cS_Ga}\}\bigr\}\cup\{\{\cS_G\}\}\bigr)\\
\notag &\biggl(
\underset{
\eqref{eq:nonexceptionaldirectionpositivemeasure},\eqref{eq:nonexceptionaldirectionzeromeasure}\&\eqref{eq:countablysupp}}{=}
1-s_\omega\nu_f\bigl(\sP^1\setminus U_{\widevec{\cS_Ga}}\bigr)-\omega(\{\{\cS_G\}\})\biggr)\\
&=\bigl(s_\omega\nu_f(U_{\widevec{\cS_Ga}})+(1-s_\omega)\bigr)
-\omega(\{\{\cS_G\}\}).\label{eq:excepdirecmeas}
\end{align}

{\bfseries (c.3).} 
Let us also see the desired estimate on $\omega(\{\{\cS_G\}\})$. For $n\gg 1$, 
recalling $f^n(\cS_G)\in(\cS_G,a]$, we compute
\begin{align}
\notag0\le&\omega_n\bigl(\{U_{\widevec{f^n(\cS_G)\cS_G}}\cap U_{\widevec{\cS_Ga}}\}\bigr)\\
\notag &\Biggl(=\omega_n\bigl(\bigl\{V\in S(\Gamma_n):V\subset U_{\widevec{f^n(\cS_G)\cS_G}}\bigr\}\bigr)\\
\notag&\hspace*{50pt}-\omega_n\bigl(\bigl\{U_{\vec{v}}\in S(\Gamma_n):\vec{v}\in(T_{\cS_G}\sP^1)\setminus\{\widevec{\cS_Ga}\}\bigr\}\cup\{\{\cS_G\}\}\bigr)\\
\notag&\underset{\eqref{eq:eventualfull}\&\eqref{eq:excepdirecmeas}}{\equiv}
s_\omega+\Bigl(\bigl(s_\omega\nu_f(U_{\widevec{\cS_Ga}})+(1-s_\omega)\bigr)
-\omega(\{\{\cS_G\}\})-1\Bigr)\Biggr)\\
&= s_\omega\nu_f\bigl(U_{\widevec{\cS_Ga}}\bigr)-\omega(\{\{\cS_G\}\}),
\label{eq:positivethick}
\end{align}
which yields the upper bound
$\omega(\{\{\cS_G\}\})\le s_\omega\nu_f(U_{\widevec{\cS_Ga}})$. 
Moreover, for $n\gg 1$, 
by 
\eqref{eq:annulus},
\eqref{eq:countablysupp},
\eqref{eq:pullbacknonexcep}, and 
\eqref{eq:nonexceptionaldirectionzeromeasure}, 
we have
\begin{gather}
\omega_n\bigl(\bigl\{U_{\vec{w}}\in S(\Gamma_n):\vec{w}\in(T_{f^n(\cS_G)}\sP^1)\setminus\{\widevec{f^n(\cS_G)a},\widevec{f^n(\cS_G)\cS_G}\}\bigr\}\bigr)
=0,\label{eq:annulusmeas}
\end{gather}
and 
using the equality
$(\pi_{\Gamma_n,\Gamma_G})_*\omega_n=\omega$ in $M^1(\Gamma_G)$
(and \eqref{eq:factorrel}), we also have
\begin{multline}
 \omega_n\bigl(\{V\in S(\Gamma_n):V\subset U_{\widevec{\cS_Ga}}\}\bigr)
=\omega(\{U_{\widevec{\cS_Ga}}\})\\
\underset{\eqref{eq:excepdirecmeas}}{=}\bigl(s_\omega\nu_f(U_{\widevec{\cS_Ga}})+(1-s_\omega)\bigr)
-\omega(\{\{\cS_G\}\}).\label{eq:fineexcepdirecmeas}
\end{multline}
Then for $n\gg 1$, we compute
\begin{align*}
0\le&\omega_n\bigl(\{U_{\widevec{f^n(\cS_G)a}}\}\bigr)\\
=&
\omega_n\bigl(\{V\in S(\Gamma_n):V\subset U_{\widevec{\cS_Ga}}\}\bigr)
-\omega_n\bigl(\{U_{\widevec{f^n(\cS_G)\cS_G}}\cap U_{\widevec{\cS_Ga}}\}\bigr)\\
&-\omega_n\bigl(\bigl\{U_{\vec{w}}:\vec{w}\in(T_{f^n(\cS_G)}\sP^1)\setminus\{\widevec{f^n(\cS_G)a},\widevec{f^n(\cS_G)\cS_G}\}\bigr\}\bigr)
-\omega_n(\{\{f^n(\cS_G)\}\})\\
&\underset{\eqref{eq:fineexcepdirecmeas},
\eqref{eq:positivethick},\eqref{eq:annulusmeas}\&\eqref{eq:gaussimagemeas}}{=}
(1-s_\omega)-\frac{\omega(\{\{\cS_G\}\})}{1-\nu_f(U_{\widevec{\cS_Ga}})},
\end{align*}
which yields the other upper bound
$\omega(\{\{\cS_G\}\})\le(1-s_\omega)(1-\nu_f(U_{\widevec{\cS_Ga}}))$.
Hence $\Delta_f$ is contained in the right hand side in \eqref{eq:computationdelta}.

{\bfseries (c.4).}
Conversely, pick $\omega$ in the right hand side
in \eqref{eq:computationdelta}, so that 
for some $s\in[0,1]$ and 
some $s'\in[0,\min\{s\nu_f(U_{\widevec{\cS_Ga}}),(1-s)(1-\nu_f(U_{\widevec{\cS_Ga}}))\}]$, we have
\begin{gather*}
\omega(\{U_{\vec{v}}\})=s\nu_f(U_{\vec{v}})
\quad\text{for every }\vec{v}\in(T_{\cS_G}\sP^1)\setminus\{\widevec{\cS_Ga}\},\\
\omega(\{\{\cS_G\}\})=s',\quad\text{and}\\
\omega\bigl(\{U_{\widevec{\cS_Ga}}\}\bigr)=\bigl(s\nu_f(U_{\widevec{\cS_Ga}})+(1-s)\bigr)-s'.
\end{gather*}
For $n\gg 1$, recalling that $\lim_{n\to\infty}f^n(\cS_G)=a$, that
$f^n(\cS_G)\in(\cS_G,a]$, and that 
$\nu_f(U_{\widevec{\cS_Ga}})<1$,
there is 
$\omega_n\in M^1(\Gamma_n)$ 
such that
\begin{gather*}
\begin{cases}
\omega_n(\{\{\cS_G\}\})=s',\\
\displaystyle\omega_n(\{\{f^n(\cS_G)\}\})=\frac{s'}{1-\nu_f(U_{\widevec{\cS_Ga}})},\\
\omega_n(\{U_{\vec{v}}\})
=\begin{cases}
  s\nu_f(U_{\vec{v}})&
  \text{for every }\vec{v}\in(T_{\cS_G}\sP^1)\setminus\{\widevec{\cS_Ga}\},\\
  0 &
  \text{for every }\vec{v}\in(T_{f^n(\cS_G)}\sP^1)\setminus\{\widevec{f^n(\cS_G)a},\widevec{f^n(\cS_G)\cS_G}\}
 \end{cases}\\
\omega_n\bigl(\{U_{\widevec{\cS_Ga}}\cap U_{\widevec{f^n(\cS_G)\cS_G}}\}\bigr)=s\nu_f(U_{\widevec{\cS_Ga}})-s'(\ge 0),\quad\text{and}\\
\displaystyle\omega_n\bigl(\{U_{\widevec{f^n(\cS_G)a}}\}\bigr)
=1-s-\frac{s'}{1-\nu_f(U_{\widevec{\cS_Ga}})}(\ge 0)
\end{cases}
\end{gather*}
(indeed $\omega_n\ge 0$ and
$\omega_n(\sP^1/S(\Gamma_n))=1-s+s\nu_f(\sP^1\setminus\{\cS_G\})\underset{\eqref{eq:vanish}}{=}1$) and that $\omega_n(S(\Gamma_n)\setminus F)=0$ for some countable
subset $F$ in $S(\Gamma_n)$ (by \eqref{eq:countablysupp}). Then for $n\gg 1$,
we have $(\pi_{\Gamma_n,\Gamma_G})_*\omega_n=\omega$
in $M^1(\Gamma_G)$ (also by \eqref{eq:factorrel}).
Moreover, for $n\gg 1$, recalling Definition \ref{th:multfactor},
{\bfseries (I)} for every 
$\vec{v}\in(T_{\cS_G}\sP^1)\setminus\{\widevec{\cS_Ga}\}$
satisfying $\nu_f(U_{\vec{v}})>0$, we have
\begin{multline*}
 \bigl(d^{-n}(f^n)_G^*\omega_n\bigr)(\{U_{\vec{v}}\})\\
\underset{\eqref{eq:pullbackdefining},\eqref{eq:nonexceptionaldirection}\&\eqref{eq:nonexceptionalpositive}}{=}
\frac{m_{\vec{v}}(f^n)\cdot
\omega_n\bigl(\bigl\{V\in S(\Gamma_n):V\subset U_{\widevec{f^n(\cS_G)\cS_G}}\bigr\}\bigr)}{d^n}\\
\underset{\eqref{eq:nonexceptionaldirectionmeasure}\&\eqref{eq:countablysupp}}{=}
\nu_f(U_{\vec{v}})\cdot s\nu_f\bigl(\sP^1\setminus\{\cS_G\}\bigr)
\underset{\eqref{eq:vanish}}{=}
s\nu_f(U_{\vec{v}})=\omega(\{U_{\vec{v}}\}),
\end{multline*}
{\bfseries (II)} for every $\vec{v}\in(T_{\cS_G}\sP^1)\setminus\{\widevec{\cS_Ga}\}$ satisfying $\nu_f(U_{\vec{v}})=0$,
\begin{multline*}
 \bigl(d^{-n}(f^n)_G^*\omega_n\bigr)(\{U_{\vec{v}}\})
\underset{\eqref{eq:pullbackdefining}\&\eqref{eq:nonexceptionaldirection}}{=}
\frac{m_{\vec{v}}(f^n)
\cdot\omega_n\bigl(\{U_{(f^n)_*(\vec{v})}\}\bigr)}{d^n}
\underset{\eqref{eq:annulus}}{=}0=s\nu_f(U_{\vec{v}})=\omega(\{U_{\vec{v}}\}),
\end{multline*} 
and {\bfseries (III)}
\begin{multline*}
 \bigl(d^{-n}(f^n)_G^*\omega_n\bigr)(\{\{\cS_G\}\})
\underset{\eqref{eq:pullbackdefining}}{=}\frac{\deg_{\cS_G}(f^n)\cdot\omega_n(\{\{f^n(\cS_G)\}\})}{d^n}\\
\underset{\eqref{eq:stationary}}{=}
\bigl(1-\nu_f(U_{\widevec{\cS_Ga}})\bigr)\cdot\omega_n(\{\{f^n(\cS_G)\}\})
=s'=\omega(\{\{\cS_G\}\}),
\end{multline*}
and then 
\begin{multline*}
 (d^{-n}(f^n)_G^*\omega_n)(\{U_{\widevec{\cS_Ga}}\})
=1-(d^{-n}(f^n)_G^*\omega_n)(S(\Gamma_G)\setminus\{U_{\widevec{\cS_Ga}}\})\\
=1-\omega(S(\Gamma_G)\setminus\{U_{\widevec{\cS_Ga}}\})
=\omega(\{U_{\widevec{\cS_Ga}}\}).
\end{multline*}
Hence for $n\gg 1$, we also have
$d^{-n}(f^n)_G^*\omega_n=\omega$ in $M^1(\Gamma_G)$, and
the right hand side in \eqref{eq:computationdelta}
is contained in $\Delta_f$.

{\bfseries (d).} 
Once the equality \eqref{eq:computationdelta} is at our disposal,
the final assertion in the case (ii) in Theorem \ref{th:computation} 
(under the assumption that $f$ is tamely maximally ramified near $a$)
is clear also recalling Remark \ref{th:notalways}.
Now the proof of Theorem \ref{th:computation} is complete.
\end{proof}

\section{Proof of Theorem \ref{th:B}}
\label{sec:proof}

We use the notations in Sections \ref{sec:quantizedbalanced} 
and \ref{sec:direct}. Let 
\begin{gather*}
 f\in\bigl(\cO(\bD)[t^{-1}]\bigr)(z)(\subset\bL(z))
\end{gather*}
be a meromorphic family of rational functions on $\bP^1(\bC)$ of degree $d>1$,
and suppose that $f^{-1}(\cS_G)\neq\{\cS_G\}$ in $\sP^1(\bL)$.
Then $f^{-n}(\cS_G)\neq\{\cS_G\}$ for every $n\in\bN$
(see Subsection \ref{sec:balanced}).
Recall that $\Char\bL=\Char k_{\bL}=\Char\bC=0$ 
and that the absolute value $|\cdot|_r$ on $\bL$
is (the extension of) \eqref{eq:normLaurent},
fixing $r\in(0,1)$ once and for all. Since $\nu_{f^2}=\nu_f$ on $\sP^1(\bL)$,
$\mu_{(f_t)^2}=\mu_{f_t}$ on $\bP^1(\bC)$ for every $t\in\bD^*$,
$E(f^2)=E(f)$, and $\#E(f)\le 2$, replacing $f$ with $f^2$ if necessary,
we can assume that $f(a)=a$ or equivalently $f^{-1}(a)=\{a\}$ for any $a\in E(f)$ 
with no loss of generality.

Recall that 
\begin{gather*}
 \Gamma_G:=\{\cS_G\}\quad\text{and}\quad\Gamma_n=\Gamma_{f^n}:=\bigl\{\cS_G,f^n(\cS_G)\bigr\}\quad\text{in }\sH^1_{\mathrm{II}}(\bL)
\end{gather*}
for every $n\in\bN$, and that
$M^1(\Gamma_G)^{\dag}$ is identified with $M^1(\bP^1(\bC))^{\dag}$
under the bijection $S(\Gamma_G)\setminus\{\cS_G\}
=T_{\cS_G}\sP^1(\bL)\cong\bP^1(k_{\bL})=\bP^1(\bC)$. 
For every $n\in\bN$, pick a meromorphic family
\begin{gather*}
 A_n\in\bigl(\cO(\bD)[t^{-1}]\bigr)(z)
\end{gather*}
of M\"obius transformations on $\bP^1(\bC)$ such that $(A_n\circ f^n)(\cS_G)=\cS_G$
in $\sP^1(\bL)$
(by the existence part of Theorem \ref{th:postfamily}).

Let 
\begin{gather*}
 \mu_0=\lim_{j\to\infty}\mu_{f_{t_j}} 
\end{gather*}
be any weak limit point of $(\mu_{f_t})_{t\in\bD^*}$ on $\bP^1(\bC)$ as $t\to 0$,
where the sequence $(t_j)$ in $\bD^*$ tends to $0$ as $j\to\infty$. 
Then taking a subsequence of $(t_j)$ if necessary, for any $n\in\bN$,
there also exists the weak limit 
\begin{gather*}
 \mu_E^{(n)}:=\lim_{j\to\infty}\bigl((A_n)_{t_j}\bigr)_*\mu_{f_{t_j}}\quad\text{on }\bP^1(\bC). 
\end{gather*}
For every $n\in\bN$, by 
Theorem \ref{th:balancedgeneral} and Proposition \ref{th:admissible},
the ordered pair 
\begin{gather*}
 \mu^{(n)}:=\bigl(\mu_0,\mu_E^{(n)}\bigr)\in\bigl(M^1(\bP^1(\bC))^\dag\bigr)^2
\end{gather*}
not only has the degenerating $f^n$-balanced property (the former half
in \eqref{eq:reductionconstant})
but also satisfies the admissibility \eqref{eq:compatibility} (for $A_n$), and in turn
also by Proposition \ref{th:transfer}, we have
\begin{gather*}
 \omega_0:=(\pi_{\Gamma_n,\Gamma_G})_*\omega_{\mu^{(n)}}\in\Delta_f^\dag;
\end{gather*}
this measure $\omega_0$ is indeed independent of $n\in\bN$, and is identified
with $\mu_0$ under the identification of 
$M^1(\Gamma_G)^{\dag}$ with $M^1(\bP^1(\bC))^{\dag}$ (by \eqref{eq:converse}).

Hence 
in the case (i) 
in Theorem \ref{th:computation},
we have the desired
$\mu_0(=\omega_0)=(\pi_{\Gamma_G})_*\nu_f$ in $M^1(\Gamma_G)^{\dag}=M^1(\bP^1(\bC))^\dag$.

{\bfseries (a).} Suppose now that 
the case (ii) in Theorem \ref{th:computation} occurs.
Then there is $a=a(t)\in E(f)(\subset\bP^1(\bL))$ such 
that $\lim_{n\to\infty}f^n(\cS_G)=a$ and 
that $f^n(\cS_G)\in(\cS_G,a]$ for $n\gg 1$,
and then $\deg_{f^n(\cS_G)}(f)\equiv d$ for $n\gg 1$;
since $\nu_{f^n}=\nu_f$ on $\sP^1(\bL)$ for every $n\in\bN$, 
$\mu_{(f_t)^n}=\mu_{f_t}$ on $\bP^1(\bC)$ 
for every $t\in\bD^*$ and every $n\in\bN$, and $E(f^n)=E(f)$ for every $n\in\bN$, 
replacing $f$ with $f^\ell$ for some $\ell\gg 1$ if necessary, 
we also assume that for every $n\in\bN$,
$f^n(\cS_G)\in(\cS_G,a]$ (so $\Gamma_n\neq\Gamma_G$),
$\deg_{f^n(\cS_G)}(f)\equiv d$,
and both \eqref{eq:superattbasin} and \eqref{eq:stationary} hold, 
with no loss of generality. 

{\bfseries (b).}
Set
\begin{gather*}
 B_1(z):=
\begin{cases}
 \frac{1}{z-a} & \text{if }a\in\cO_{\bL},\\
 \frac{-z}{(z/a)-1} & \text{if }a\in\bL\setminus\cO_{\bL},\\
 z & \text{if }a=\infty\in\bP^1(\bL)(=\bL\cup\{\infty\})
\end{cases}
\in\PGL(2,\cO_{\bL}),
\end{gather*}
so that $B_1(a)=\infty$ and that $B_1(\cS_G)=\cS_G$ (or equivalently 
$\widetilde{B_1}=\phi_{\widetilde{B_1}}\in\PGL(2,k_{\bL})=\PGL(2,\bC)$,
and then $\widetilde{B_1^{-1}}=\phi_{\widetilde{B_1^{-1}}}=\phi_{\widetilde{B_1}}^{-1}\in\PGL(2,k_{\bL})=\PGL(2,\bC)$), and set 
\begin{gather*}
 f_{B_1}:=B_1\circ f\circ {B_1}^{-1}\in\bL[z].
\end{gather*}

{\bfseries (c.1).} 
Write $f_{B_1}(z)=\sum_{j=0}^dc_j(t)z^j\in\bL[z]$ (so 
$c_d\in\bL\setminus\{0\}$)
and set
\begin{gather*}
 d_0:=\max\Bigl\{j\in\{0,1,\ldots,d\}:|c_j|_r
=\max_{i\in\{0,1,\ldots,d\}}|c_i|_r\Bigr\}. 
\end{gather*}
Then noting that $f_{B_1}(\cS_G)\in(\cS_G,\infty]$, we have 
$|c_{d_0}|_r>1$, 
and 
$f_{B_1}(\cS_G)$ is represented by (the constant sequence of) 
the $\bL$-closed disk $B(0,|c_{d_0}|_r)$.
Setting
\begin{gather*}
 B_2(z):=c_{d_0}^{-1}z\in\bL[z]\cap\PGL(2,\bL),
\end{gather*}
so that $(B_2\circ f_{B_1})(\cS_G)=\cS_G$, 
we have
$\phi_{\widetilde{B_2\circ f_{B_1}}}(\zeta) =\sum_{j=0}^{d_0}\widetilde{(\frac{c_j}{c_{d_0}})}\cdot\zeta^j$,
\begin{multline}
 d_0=\deg(\phi_{\widetilde{B_2\circ f_{B_1}}})
\underset{\eqref{eq:totallocaldegree}}{=}\deg_{\cS_G}(B_2\circ f_{B_1})=\\
=\deg_{f(\cS_G)}(B_2\circ B_1)\cdot\deg_{B_1^{-1}(\cS_G)}(f)\cdot\deg_{\cS_G}(B_1^{-1})=\deg_{\cS_G}(f)(>0),\label{eq:locdeggauss}
\end{multline}
and ($H_{\widetilde{B_2\circ f_{B_1}}}(\zeta_0,\zeta_1)=\zeta_0^{d-d_0}$, so in particular)
\begin{gather}
\ord_{\zeta=\infty}\bigl[H_{\widetilde{B_2\circ f_{B_1}}}=0\bigr]=d-d_0=d-\deg_{\cS_G}(f).
\label{eq:holespoles}
\end{gather}

{\bfseries (c.2).} For each $j\in\{0,\ldots,d\}$, set
\begin{gather*}
C_j=C_j(t):=\frac{c_j}{c_{d_0}}\cdot c_{d_0}^{j-d_0}\in\bL,
\quad\text{so that }C_{d_0}=1\text{ and that }|C_j|_r<1
\text{ if }j<d_0,
\end{gather*}
and also set
\begin{gather*}
 f_{B_2B_1}(w):=(B_2\circ f_{B_1}\circ B_2^{-1})(w)
=c_{d_0}^{d_0}\Biggl(w^{d_0}+\sum_{j\in\{0,1,\ldots,d\}\setminus\{d_0\}}C_jw^j\Biggr)
\in\bL[z]. 
\end{gather*}
Then using Fact \ref{th:Mobiustangent} and \eqref{eq:tangentreduct}
(for $B_2^{-1},B_2\in\PGL(2,\bL)$), we have
\begin{align} 
\notag f_{B_2B_1}(U_{\widevec{\cS_G\infty}})
\biggl(&=(B_2\circ f_{B_1})\bigl(B_2^{-1}(U_{\widevec{\cS_G\infty}})\bigr)
=(B_2\circ f_{B_1})\bigl(U_{\widevec{B_2^{-1}(\cS_G)\infty}}\bigr)\\
\notag &\underset{(B_2\circ f_{B_1})(\cS_G)=\cS_G}{=}
(B_2\circ f_{B_1})\bigl(U_{\widevec{f_{B_1}(\cS_G)\infty}}\bigr)
=B_2\bigl(f_{B_1}(U_{\widevec{f_{B_1}(\cS_G)\infty}})\bigr)\\
\notag &\underset{\eqref{eq:superattbasin}\text{ applied to }n=1}{=}
B_2\bigl(U_{\widevec{f_{B_1}^2(\cS_G)\infty}}\bigr)
=U_{\widevec{((B_2\circ f_{B_1}^2)(\cS_G))\infty}}\biggr)\\
&\subsetneq\sP^1(\bL).\label{eq:image}
\end{align}

\begin{claim}\label{th:coefficient}
Either $d_0=d$ or there is $j>d_0$ such that $|C_j|_r\ge 1$.
\end{claim}

\begin{proof}
Otherwise, $d_0<d$ and,
$|C_j|_r<1$ for every $j\in\{0,\ldots,d\}\setminus\{d_0\}$. 
Then since $|c_{d_0}^{d_0}|_r=|c_{d_0}|_r^{d_0}>1$,
we have $H_{\widetilde{f_{B_2B_1}}}(\zeta_0,\zeta_1)=\zeta_0^{d-d_0}\zeta_1^{d_0}$
(and $\phi_{\widetilde{f_{B_2B_1}}}\equiv\infty\in\bP^1(k_{\bL})$), 
so that
$\ord_{\zeta=\infty}[H_{\widetilde{f_{B_2B_1}}}=0]=d-d_0$. In particular,
we must have 
\begin{gather*}
 s_{\widevec{\cS_G\infty}}(f_{B_2B_1})=\ord_{\zeta=\infty}[H_{\widetilde{f_{B_2B_1}}}=0]=d-d_0>0 
\end{gather*}
(by Fact \ref{th:surplusalgclosed}),
so
$f_{B_2B_1}\bigl(U_{\widevec{\cS_G\infty}}\bigr)=\sP^1(\bL)$
(by Fact \ref{th:surplus}).
This contradicts \eqref{eq:image}. 
\end{proof}

{\bfseries (c.3).}
 Since this $a\in E(f)$ is a fixed point in $\bP^1(\bL)$ of 
 $f\in(\cO(\bD)[t^{-1}])(z)$,
 this $a=a(t)$ 
 is indeed in $\bP^1(\mathbb{K})$ over a finite algebraic field
 extension $\mathbb{K}$ of the quotient field of the domain $\cO(\bD)[t^{-1}]$, 
 that is, for $0<s_0\ll 1$,
 by the substitution/change of indeterminants $t=s^m$ for some $m\in\bN$, we have 
\begin{gather*}
  a=a(s^m)\in\bP^1(\cO(\bD_{s_0})[s^{-1}])(\subset\bP^1(\bL)),
\end{gather*} 
where 
 $\bD_{s_0}:=\{s\in\bC:|s|<s_0\}$ (cf.\ \cite[Proof of Corollary 5.3]{DF14}).
 Then decreasing $0<s_0\ll 1$ if necessary, we have not only
 $c_j(s^m),C_j(s^m)\in\cO(\bD_{s_0})[s^{-1}]\subset\bL$ 
 for every $j\in\{0,1,\ldots,d\}$ but also
 $(B_1)_{s^m},(B_2)_{s^m}\in\PGL(2,\cO(\bD_{s_0})[s^{-1}])(\subset\PGL(2,\bL)$
 and indeed $(B_1)_{s^m}\in\PGL(2,\cO_{\bL})$),
 and still $(B_1)_{s^m}(\cS_G)=\cS_G$ in $\sP^1(\bL)$ or equivalently
$\widetilde{(B_1)_{s^m}}=\phi_{\widetilde{(B_1)_{s^m}}}
(=\phi_{\widetilde{B_1}}=\widetilde{B_1})$ in $\PGL(2,\bC)=\PGL(2,k_{\bL})$.

Let us, for notational simplicity, denote by
\begin{gather*}
 A:=A_1=(A_1)_t\in\bigl(\cO(\bD)[t^{-1}]\bigr)(z)
\end{gather*}
the 
meromorphic 
family $A_1$ ($A_n$ for $n=1$)
of M\"obius transformations on $\bP^1(\bC)$,
and also by
\begin{gather*}
\mu_E:=\mu_E^{(1)}\in M^1(\bC)^\dag
\quad\text{and}\quad
\mu:=\mu^{(1)}=(\mu_0,\mu_E)\in\bigl(M^1(\bP^1(\bC))^\dag\bigr)^2
\end{gather*}
the probability measure $\mu_E^{(1)}$ and the ordered pair $\mu^{(1)}$, respectively.
Set 
\begin{gather*}
 D=D_s:=(B_2\circ B_1)_{s^m}\circ(A_{s^m})^{-1}\in
\PGL\bigl(2,\cO(\bD_{s_0})[s^{-1}]\bigr)\bigl(\subset\PGL(2,\bL)\bigr),
\end{gather*}
so that $\tilde{D}=\phi_{\tilde{D}}$ 
in $\PGL(2,\bC)=\PGL(2,k_{\bL})$ 
(by the uniqueness part in Theorem \ref{th:postfamily})
since 
$((B_2\circ B_1)_{s^m}\circ f_{s^m})(\cS_G)
=((B_2\circ f_{B_1})_{s^m}((B_1)_{s^m}(\cS_G))
=(B_2\circ f_{B_1})_{s^m}(\cS_G)
=\cS_G=(A\circ f)_{s^m}(\cS_G)$.

\begin{claim}\label{th:support}
$\supp((\phi_{\tilde{D}})_*\mu_E)\subset\bP^1(\bC)\setminus\{\infty\}$.
\end{claim}

\begin{proof}
Recall that $|\cdot|_r$ and $|\cdot|$ are the absolute values on 
$\bL$ and on $\bC$, respectively.
For every $s\in\bD_{s_0}^*$ and every $z\in\bC$, we compute
\begin{gather*}
 (f_{B_1})_{s^m}\bigl(c_{d_0}(s^m)z\bigr)
 =\bigl(c_{d_0}(s^m)\bigr)^{d_0+1}z^{d_0}\cdot
 \Biggl\{1+\sum_{j\in\{0,1,\ldots,d\}\setminus\{d_0\}}
C_j(s^m)z^{j-d_0}\Biggr\}.
\end{gather*}
Let us see that for $\ell\gg 1$, if $0<|s|\ll s_0$, then
\begin{gather*}
 \inf_{|z|=\ell}\Biggl|1+\sum_{j\in\{0,1,\ldots,d\}\setminus\{d_0\}}
C_j(s^m)z^{j-d_0}\Biggr|\ge\frac{1}{2}(>0); 
\end{gather*}
for, in the latter case in Claim \ref{th:coefficient}, we set
\begin{gather*}
 d_1:=\max\bigl\{j\in\{d_0+1,\ldots,d\}:|C_j|_r=\max_{j>d_0}|C_j|_r(\ge 1)\bigr\},
\end{gather*}
so that ($d_1>d_0$, that) $\limsup_{s\to 0}|C_{d_1}(s^m)|\in(0,+\infty]$ 
(since $|C_{d_1}(s^m)|_r=|C_{d_1}|_r^m\ge 1$), and that
for every $j>d_0$, $C_j(s^m)/C_{d_1}(s^m)\in\cO(\bD_{s_0})$, which
vanishes at $s=0$ if $j>d_1$ 
(since $|C_j(s^m)/C_{d_1}(s^m)|_r=|C_j|_r^m/|C_{d_1}|_r^m$ is $\le 1$
if $j>d_0$, and is $< 1$ if $j>d_1$).
Then for $\ell\gg 1$ (so that the second and third inequalities below hold), 
if $0<|s|\ll s_0$ (so that the first and fourth ones below hold), then
\begin{align*}
&\Biggl|\sum_{j<d_0}C_j(s^m)z^{j-d_0}\Biggr|
\Biggl(\le \sum_{j<d_0}(0+1)\ell^{j-d_0}\quad(\text{by }
|C_j(s^m)|_r=|C_j|_r^m<1\text{ if }j<d_0)\\
\le &d_0\quad(\text{noting that the sum above is over }j<d_0)\\
\le&\biggl(\min\Bigl\{1,2^{-1}\cdot\limsup_{s\to 0}|C_{d_1}(s^m)|\Bigr\}\biggr)\times\\
&\hspace*{45pt}\times\biggl(\ell^{d_1-d_0}\Bigl(1-\sum_{d_1>j>d_0}\Bigl(\Bigl|\frac{C_j(s^m)}{C_{d_1}(s^m)}\Bigr||_{s=0}+1\Bigr)\ell^{j-d_1}\Bigr)-1\biggr)-\frac{3}{2}\\
\le&|C_{d_1}(s^m)|\cdot
\Bigl(\ell^{d_1-d_0}
-\sum_{d_1>j>d_0}\Bigl|\frac{C_j(s^m)}{C_{d_1}(s^m)}\Bigr|\ell^{j-d_0}-\sum_{j>d_1}\Bigl|\frac{C_j(s^m)}{C_{d_1}(s^m)}\Bigr|\ell^{j-d_0}\Bigr)-\frac{3}{2}\Biggr)\\
\le&\Biggl|\sum_{j>d_0}C_j(s^m)z^{j-d_0}\Biggl|-\frac{3}{2}
\end{align*}
on $\{z\in\bC:|z|=\ell\}$,
which yields the desired inequality in this case.
Similarly, in the former case ($d_0=d$) in Claim \ref{th:coefficient},
for $\ell\gg 1$ (so 
that the final inequality below holds), 
if $0<|s|\ll s_0$ (so 
that the second inequality below holds), then
\begin{gather*}
\biggl|\sum_{j<d_0}C_j(s^m)z^{j-d_0}\biggr|
\le\sum_{j<d_0}|C_j(s^m)|\ell^{j-d_0}\le\sum_{j<d_0}(0+1)\ell^{j-d_0}\le\frac{1}{2}
\end{gather*}
on $\{z\in\bC:|z|=\ell\}$,
which still yields the desired inequality in this case.

Hence, since $d_0\ge 1$ (in \eqref{eq:locdeggauss}) and 
$|c_{d_0}(s^m)|_r=|c_{d_0}|_r^m>1$,
fixing $\ell_0\gg 1$, if $0<|s|\ll s_0$, then
$(f_{B_1})_{s^m}\bigl(\bigl\{z\in\bC:|z|=|c_{d_0}(s^m)|\ell_0\bigr\}\bigr)
\subset\bigl\{z\in\bC:|z|\ge |c_{d_0}(s^m)|^{d_0+1}\ell_0^{d_0}/2\bigr\}
\subset\bigl\{z\in\bC:|z|\ge 2|c_{d_0}(s^m)|\ell_0\bigr\}$, 
which with the maximum modulus principle for
holomorphic functions applied to $1/((f_{B_1})_{s^m}(1/w))$ near $w=0\in\bC$
in turn yields 
\begin{gather*}
 (f_{B_1})_{s^m}\bigl(\bigl\{z\in\bC:|z|>|c_{d_0}(s^m)|\ell_0\bigr\}\bigr)\subset\bigl\{z\in\bC:|z|>2|c_{d_0}(s^m)|\ell_0\bigr\}, 
\end{gather*}
so that
$\supp(((B_1)_{s^m})_*(\mu_{f_{s^m}}))(=\supp(\mu_{(f_{B_1})_{s^m}}))\subset\{z\in\bC:|z|\le\bigl|c_{d_0}(s^m)\bigr|\ell_0\}$
(see Fact \ref{maxentropy}). Hence for $0<|s|\ll s_0$, 
recalling that $(B_2)_{s^m}(z)=(c_{d_0}(s^m))^{-1}z$, 
we have
\begin{gather*}
\supp\bigl((D_s)_*(A_{s^m})_*\mu_{f_{s^m}}\bigr)
\bigl(=\supp(((B_2\circ B_1)_{s^m})_*\mu_{f_{s^m}})\bigr)
\subset\bigl\{z\in\bC:|z|\le\ell_0\bigr\}.\label{eq:inclusion}
\end{gather*}
Recall that
$\mu_E:=\lim_{j\to\infty}(A_{t_j})_*\mu_{f_{t_j}}$ weakly on $\bP^1(\bC)$, 
and pick a sequence $(s_j)$ in $\bD^*$
so that $t_j=s_j^m$ for every $j\in\bN$.
Then $\lim_{j\to\infty}D_{s_j}=\phi_{\tilde{D}}(=\tilde{D})$ 
uniformly on $\bP^1(\bC)$ (by \eqref{eq:locunifoutside}). 
Now the above inclusion 
for $s=s_j$, $j\gg 1$, 
completes the proof of Claim \ref{th:support}, by making $j\to\infty$.
\end{proof}

{\bfseries (d).}
Recalling that
$\omega_0(=(\pi_{\Gamma_f,\Gamma_G})_*\omega_{\mu})\in\Delta_f^\dag(\subset\Delta_f)$, there are $s\in[0,1]$ and
$s'\in[0,\min\{s\nu_f(U_{\widevec{\cS_Ga}}),
(1-s)(1-\nu_f(U_{\widevec{\cS_Ga}}))\}]$ 
such that
\begin{gather*}
\begin{cases}
 \omega_0(\{U_{\vec{v}}\})=s\nu_f(U_{\vec{v}})
 \text{ for every }\vec{v}\in\bigl(T_{\cS_G}\sP^1\bigr)\setminus\{\widevec{\cS_Ga}\},\\
 \omega_0(\{\{\cS_G\}\})=s',\quad\text{and}\\
 \omega_0(\{U_{\widevec{\cS_Ga}}\})
 =\bigl(s\nu_f(U_{\widevec{\cS_Ga}})+(1-s)\bigr)-s',
\end{cases}
\end{gather*}
using the computation \eqref{eq:computationdelta} of $\Delta_f$
under the standing assumption that the case (ii) in Theorem \ref{th:computation} occurs and by $\Char k_{\bL}=0$. 
Since $\omega_0\in\Delta_f^\dag$, we first have $s'=0$.

Recalling the identification $\omega_0=\mu_0$ in $M^1(\Gamma_G)^\dag=M^1(\bP^1(\bC))^\dag$ and
the degenerating $f$-balanced property (the former half in \eqref{eq:reductionconstant}) of $\mu=(\mu_0,\mu_E)$, we compute
\begin{multline*}
\bigl(s\nu_f(U_{\widevec{\cS_Ga}})+(1-s)\bigr)-s'\\
=\omega_0(\{U_{\widevec{\cS_Ga}}\})=\mu_0(\{\tilde{a}\})
=\frac{(\widetilde{A\circ f})^*\mu_E}{d}(\{\tilde{a}\})
=\frac{\bigl((\phi_{\widetilde{A\circ f}})^*\mu_E+[H_{\widetilde{A\circ f}}=0]\bigr)(\{\tilde{a}\})}{d},
\end{multline*}
and moreover, 
recalling 
that $\tilde{D}=\phi_{\tilde{D}},\tilde{B_1}=\phi_{\tilde{B_1}}\in\PGL(2,\bC)$,
that $a=B_1^{-1}(\infty)$, that $(B_2\circ f_{B_1})(\infty)=\infty$,
and that $\deg(\phi_{\widetilde{B_2\circ f_{B_1}}})=d_0>0$ (in \eqref{eq:locdeggauss}) 
and using Claim \ref{th:support}, we compute
\begin{multline*}
 \bigl((\phi_{\widetilde{A\circ f}})^*\mu_E\bigr)(\{\tilde{a}\})
=\bigl((\phi_{\widetilde{(D^{-1}\circ B_2\circ B_1\circ f\circ B_1^{-1})}})^*\mu_E\bigr)
(\{\infty\})\\
=\bigl((\phi_{\widetilde{B_2\circ f_{B_1}}})^*(\phi_{\tilde{D}})_*\mu_E\bigr)(\{\infty\})
=\bigl(\deg_{\infty}(\phi_{\widetilde{B_2\circ f_{B_1}}})\bigr)\cdot
\bigl((\phi_{\tilde{D}})_*\mu_E\bigr)(\{\infty\})=0,
\end{multline*}
and on the other hand, we compute
%, also under the normalization on $f$ in the step {\bfseries (a)}, 
\begin{multline*}
\ord_{\zeta=\tilde{a}}\bigl[H_{\widetilde{A\circ f}}=0\bigr]
\Bigl(
\underset{\eqref{eq:surplusfaber}}{=}s_{\widevec{\cS_Ga}}(f)
\underset{\eqref{eq:argumentnonarchi}}{=}s_{(B_1)_*(\widevec{\cS_Ga})}(f_{B_1})
\underset{\eqref{eq:tangentreduct}}{=}s_{\widevec{\cS_G\infty}}(f_{B_1})=
\Bigr)\\
\underset{\eqref{eq:surplusfaber}}{=}\ord_{\zeta=\infty}\bigl[H_{\widetilde{B_2\circ f_{B_1}}}=0\bigr]
\underset{\eqref{eq:holespoles}}{=}d-\deg_{\cS_G}(f)
\underset{\eqref{eq:stationary}\text{ for }n=1}{=}d\cdot\nu_f(U_{\widevec{\cS_Ga}}).
\end{multline*}
Hence we also have $s'=(1-s)(1-\nu_f(U_{\widevec{\cS_Ga}}))$.

Consequently, we have not only $s'=0$ but also 
$s=1$ since $\nu_f(U_{\operatorname{\widevec{\cS_Ga}}})<1$
(which is a consequence of \eqref{eq:stationary})
in the case (ii) in Theorem \ref{th:computation}.
Then we still have the desired 
$\mu_0=\omega_0=(\pi_{\Gamma_G})_*\nu_f$ in $M^1(\bP^1(\bC))^\dag
=M^1(\Gamma_G)^\dag$ (also by \eqref{eq:projectfactor}). 

Now the proof of Theorem \ref{th:B} is complete.
\qed

\begin{remark}\label{th:remedy}
The arguments in the steps {\bfseries (c.1, 2, 3)} in the proof of Theorem \ref{th:B} 
relate the non-archimedean absolute value $|\cdot|_r$ on $\bL$,
which restricts to the trivial absolute value on $\bC=k_{\bL}$, with the Euclidean absolute value
$|\cdot|$ on $\bC$ and complement \cite[Proof of Theorem B]{DF14}.
The final assertion in \cite[Corollary 5.3]{DF14},
which \cite[Proof of Theorem B]{DF14} is based on,
was shown in \cite{DF14} under the condition \eqref{eq:identically}
(see also Remark \ref{th:notalways}).
\end{remark}

\section{Examples}\label{sec:example}
Pick a meromorphic family 
\begin{gather*}
 f(z)=z^2+t^{-1}z\in\bigl(\cO(\bD)[t^{-1}]\bigr)[z](\subset\bL[z]) 
\end{gather*}
of quadratic polynomials on $\bP^1(\bC)$. 
Then $f^{-1}(\infty)=\{\infty\}=E(f)$,
and the case (ii) (for $a=\infty$) in Theorem \ref{th:computation} occurs
(indeed $(\cS_G,a]\ni f^n(\cS_G)=\cS_{B(0,|t^{-2^{n-1}}|_r)}\to\infty$ 
as $n\to\infty$ 
since $\cS_G$ is represented by (the constant sequence of) the
$\bL$-closed disk $\cO_K=B(0,1)$, $f(0)=0$,
$|f(1)|_r=|t^{-1}|_r(>1)$, $|f(t^{-1})|_r=|t^{-2}|_r>|t^{-1}|_r$, and
$|f(z)|_r=|z|_r^2$ on $\bL\setminus B(0,|t^{-1}|_r)$;
see \eqref{eq:normLaurent} 
for the absolute value $|\cdot|_r$ on $\bL$). Since $f'(z)=2z+t^{-1}\in\bL[z]$, 
the point $-t^{-1}+1\in U_{\widevec{\cS_G\infty}}\cap\bL$
is a (classical) repelling fixed point of $f$
(indeed $f(-t^{-1}+1)=-t^{-1}+1$ and $|f'(-t^{-1}+1)|_r=|t^{-1}|_r>1$), 
which is in $\sJ(f)=\supp\nu_f$, so we in particular have
$\nu_f(U_{\widevec{\cS_G\infty}})>0$.
Hence 
\eqref{eq:identically} in Remark \ref{th:notalways}
is not the case for this $f$. 

\section{A complement of Proposition \ref{th:transfer}}
\label{sec:complement}

Let us continue to use the notations in Sections \ref{sec:quantizedbalanced} 
and \ref{sec:direct}. 
Let 
\begin{gather*}
 f\in\bigl(\cO(\bD)[t^{-1}]\bigr)(z)(\subset\bL(z))
\end{gather*}
be a meromorphic family 
of rational functions on $\bP^1(\bC)$ of degree $d>1$,
and suppose that $f^{-1}(\cS_G)\neq\{\cS_G\}$ in $\sP^1(\bL)$.
Recall that $\Gamma_G:=\{\cS_G\}$ and $\Gamma_n:=\Gamma_{f^n}:=\{\cS_G,f^n(\cS_G)\}$
in $\sH^1_{\mathrm{II}}(\bL)$ for every $n\in\bN$ and that
$M^1(\Gamma_G)^{\dag}$ is identified with $M^1(\bP^1(\bC))^{\dag}$
under the bijection $S(\Gamma_G)\setminus\{\cS_G\}
=T_{\cS_G}\sP^1(\bL)\cong\bP^1(k_{\bL})=\bP^1(\bC)$. 

For every $n\in\bN$, pick a meromorphic family
$A_n\in(\cO(\bD)[t^{-1}])(z)$ 
of M\"obius transformations on $\bP^1(\bC)$
such that $(A_n\circ f^n)(\cS_G)=\cS_G$ in $\sP^1(\bL)$
(by Theorem \ref{th:postfamily}), and set 
\begin{gather*}
 A:=A_1.
\end{gather*}
We note that for any $\mu=(\mu_C,\mu_E)\in (M^1(\bP^1(\bC))^\dag)^2$
satisfying the admissibility \eqref{eq:compatibility} (for this $A$),
we still have $\omega_{\mu}\in M^1(\Gamma_f)^\dag$
(and $\omega_{\mu}(S(\Gamma_f)\setminus F)=0$ for some countable subset $F$
in $S(\Gamma_f)$). 

Conversely, 
for every $\omega\in M^1(\Gamma_f)^{\dag}$ satisfying
$\omega(S(\Gamma_f)\setminus F)=0$ for some countable subset $F$
in $S(\Gamma_f)$,
there is a unique ordered pair
\begin{gather*}
 \mu_\omega=(\mu_{\omega,C},\mu_{\omega,E})\in\bigl(M^1(\bP^1(\bC))^{\dag}\bigr)^2 =\bigl(M^1(\Gamma_G)^{\dag}\bigr)^2
\end{gather*}
such that
when $\Gamma_f=\Gamma_G$($\Leftrightarrow\tilde{A}=\phi_{\tilde{A}}$), 
\begin{gather*}
\begin{cases}
 \mu_{\omega,C}:=(\pi_{\Gamma_f,\Gamma_G})_*\omega\in M^1(\Gamma_G)^\dag=M^1(\bP^1(\bC))^{\dag},\\
 \mu_{\omega,E}:=\tilde{A}_*(\pi_{\Gamma_f,\Gamma_G})_*\omega
=\tilde{A}_*\mu_{\omega,C}\in M^1(\Gamma_G)^\dag=M^1(\bP^1(\bC))^{\dag}
\end{cases}
\end{gather*}
and, when $\Gamma_f\neq\Gamma_G$,
noting that $\{f(\cS_G)\}\subset\Gamma_f\subset\sH^1_{\mathrm{II}}(\bL)$,
\begin{gather*}
\begin{cases}
 \mu_{\omega,C}(\{\tilde{x}\})=\bigl((\pi_{\Gamma_f,\Gamma_G})_*\omega\bigr)\bigl(\{U_{\widevec{\cS_Gx}}\}\bigr)
 &\text{for every }\tilde{x}\in\bP^1(k_{\bL})=\bP^1(\bC),\\
 \mu_{\omega,E}(\{\tilde{y}\})=\bigl((\pi_{\Gamma_f,\{f(\cS_G)\}})_*\omega\bigr)\bigl(\bigl\{U_{(A^{-1})_*(\widevec{\cS_Gy})}\bigr\}\bigr)
 &\text{for every }\tilde{y}\in\bP^1(k_{\bL})=\bP^1(\bC);
\end{cases}
\end{gather*}
then this ordered pair $\mu_\omega=(\mu_{\omega,C},\mu_{\omega,E})$ satisfies the admissibility
\eqref{eq:compatibility} (for $A$)
(by Lemma \ref{th:annulus} when $\Gamma_f\neq\Gamma_G$), and in turn
we have both
\begin{gather}
\omega_{\mu_\omega}=\omega\quad\text{in }M^1(\Gamma_f)^\dag
\quad\text{and}\quad
\mu_{\omega_\mu}=\mu\quad\text{in }\bigl(M^1(\bP^1(\bC))^{\dag}\bigr)^2,\label{eq:bijective}
\end{gather} 
that is, the map $ (M^1(\bP^1(\bC))^\dag)^2\ni\mu\mapsto\omega_{\mu}\in M^1(\Gamma_f)^\dag$
is bijective.

We conclude with the following complement of Proposition \ref{th:transfer}.

\begin{proposition}[cf.\ {\cite[Proposition 5.1 and Theorem 5.2]{DF14}}]\label{th:complement}
There is the bijection
\begin{multline*}
 \Bigl\{(\mu_C,\mu_E)\in\bigl(M^1(\bP^1(\bC))^\dag\bigr)^2:
\text{satisfying the admissibility \eqref{eq:compatibility} $($for $A)$ and}\\
\text{the degenerating }f\text{-balanced property }(\widetilde{A\circ f})^*\mu_E=d\cdot\mu_C\text{ in }M(\sP^1)\Bigr\}\ni\mu\\
\mapsto\omega_\mu\in\Bigl\{\omega\in M^1(\Gamma_f)^\dag:\text{ satisfying }
\omega(S(\Gamma_f)\setminus F)=0 \text{ for some countable subset }F\\
\text{ in }S(\Gamma_f)\text{ and }
f_G^*\omega=d\cdot(\pi_{\Gamma_f,\Gamma_G})_*\omega\text{ in }M(\Gamma_G)\Bigr\},
\end{multline*}
the inverse of which is given by the map $\omega\mapsto\mu_\omega$.
This bijection induces the bijection 
\begin{gather*}
 \Delta_0^\dag\ni\mu_C\mapsto(\pi_{\Gamma_n,\Gamma_G})_*\bigl(\omega_{(\mu_C,\mu_E^{(n)})}\bigr)\in\Delta_f^\dag,
\end{gather*}
where
\begin{multline*}
\Delta_0^\dag:=\Bigl\{\mu_C\in M^1(\bP^1(\bC))^\dag:
\text{for $($any$)$ }n\gg 1,\\
\text{ there is }\mu_E^{(n)}\in M^1(\bP^1(\bC))^\dag
\text{ such that }
(\widetilde{A_n\circ f^n})^*\mu_E^{(n)}=d\cdot\mu_C\Bigr\}.
\end{multline*}
\end{proposition}

\begin{proof}
 The former assertion follows from 
 \eqref{eq:bijective} and the computations in the steps 
 {\bfseries (a-1)} and {\bfseries (b-1)} in the proof of
 Proposition \ref{th:transfer}. Then the latter assertion
holds also by \eqref{eq:converse}.
\end{proof}

\begin{acknowledgement}
The author thanks the referee for a very careful scrutiny and invaluable comments.
The author was partially supported by JSPS Grant-in-Aid 
for Scientific Research (C), 19K03541,
Institut des Hautes \'Etudes Scientifiques, and
F\'ed\'eration de recherche math\'ematique des Hauts-de-France
(FR CNRS 2956). The author was a long term researcher
of RIMS, Kyoto University in the period of April 2019 - March 2020, and
also thanks the hospitality there.
\end{acknowledgement} 

\def\cprime{$'$}

\end{document}